\documentclass[preprint]{elsarticle}

\usepackage{latexsym}     % Pour certains symboles (\Box, ...)
\usepackage{amsmath}      % Environnement AMS
\usepackage{amsthm}       % avec theoremes
\usepackage{amsfonts}     % fonts associees
\usepackage{amssymb}      % et symboles
\usepackage{euscript}
\usepackage{stmaryrd}

\usepackage[english]{babel}

\usepackage{algorithm,algorithmic}
\usepackage{version}

\def\Figures{./}

\usepackage{ifpdf}
\ifpdf
  \graphicspath{{\Figures}}
  \usepackage[pdftex,hypertexnames=false]{hyperref}
\else
  \graphicspath{{\Figures}}
  \usepackage[dvips,hypertexnames=false]{hyperref}
\fi

\usepackage{xcolor}

\usepackage{todonotes}

\newcommand{\dc}[1]{\llbracket #1 \rrbracket}

\usepackage{tikz}
\usetikzlibrary{shapes,arrows}

\theoremstyle{plain}

\newtheorem{proposition}{Proposition}

\theoremstyle{definition}

\newtheorem{assumption}{Assumption}

\theoremstyle{remark}
\newtheorem{remark}{Remark}

\newcommand{\coutint}{L}                                    % Co\^{u}t int\'{e}gral
\newcommand{\coutfin}{K}                                    % Co\^{u}t final
\newcommand{\production}{\Theta}
\newcommand{\dynamics}{f}                                   % Dynamique

\def\mathscr{\EuScript}

\newcommand{\tribu}[1]{\mathscr{#1}}                        % Tribu
\newcommand{\omeg}{\Omega}                                  % espace du triplet
\newcommand{\trib}{\tribu{A}}                               % tribu  du triplet
\newcommand{\prbt}{\mathbb{P}}                              % proba  du triplet
\newcommand{\espe}{\mathbb{E}}                              % Symbole esp\'{e}rance
\def\va#1{{\boldsymbol{\uppercase{#1}}}}

\newcommand{\np}[1]{(#1)}                                   % Parenth\`{e}se normal
\newcommand{\bp}[1]{\big(#1\big)}                           % Parenth\`{e}se big
\newcommand{\Bp}[1]{\Big(#1\Big)}                           % Parenth\`{e}se Big
\newcommand{\bgp}[1]{\bigg(#1\bigg)}                        % Parenth\`{e}se bigg
\newcommand{\Bgp}[1]{\Bigg(#1\Bigg)}                        % Parenth\`{e}se Bigg
\newcommand{\na}[1]{\{#1\}}                                 % Accolade normal
\newcommand{\ba}[1]{\big\{#1\big\}}                         % Accolade big
\newcommand{\normdelim}[1]{\np{#1}}                         % Taille ``normal''
\newcommand{\bigdelim}[1]{\bp{#1}}                          % Taille ``big''
\newcommand{\Bigdelim}[1]{\Bp{#1}}                          % Taille ``Big''
\newcommand{\biggdelim}[1]{\bgp{#1}}                        % Taille ``bigg''
\newcommand{\Biggdelim}[1]{\Bgp{#1}}                        % Taille ``Bigg''
\newcommand{\normdelims}[2]{\normdelim{#1\mid#2}}           % avec s\'{e}parateur
\newcommand{\bigdelims}[2]{\bigdelim{#1\ \big|\ #2}}        %
\newcommand{\Bigdelims}[2]{\Bigdelim{#1\ \Big|\ #2}}        %
\newcommand{\biggdelims}[2]{\biggdelim{#1\ \bigg|\ #2}}     %

\newcommand{\nesp}[2][]{\espe_{#1}\normdelim{#2}}           % Esp\'{e}rance normal
\newcommand{\besp}[2][]{\espe_{#1}\bigdelim{#2}}            % Esp\'{e}rance big
\newcommand{\Besp}[2][]{\espe_{#1}\Bigdelim{#2}}            % Esp\'{e}rance Big
\newcommand{\bgesp}[2][]{\espe_{#1}\biggdelim{#2}}          % Esp\'{e}rance bigg
\newcommand{\Bgesp}[2][]{\espe_{#1}\Biggdelim{#2}}          % Esp\'{e}rance Bigg

\newcommand{\nespc}[3][]{\espe_{#1}\normdelims{#2}{#3}}     % Esp\'{e}rance cond. normal
\newcommand{\bespc}[3][]{\espe_{#1}\bigdelims{#2}{#3}}      % Esp\'{e}rance cond. big
\newcommand{\Bespc}[3][]{\espe_{#1}\Bigdelims{#2}{#3}}      % Esp\'{e}rance cond. Big
\newcommand{\bgespc}[3][]{\espe_{#1}\biggdelims{#2}{#3}}    % Esp\'{e}rance cond. bigg
\def\eqsepv{\; , \enspace}                                  % Virgule dans une \'{e}quation
\def\eqfinv{\; ,}                                           % Virgule en fin d'\'{e}quation
\def\eqfinp{\; .}                                           % Point en fin d'\'{e}quation
\newcommand{\xit}{\va{x}_{t}^{i}}
\newcommand{\xitp}{\va{x}_{t+1}^{i}}
\newcommand{\xitf}{\va{x}_{T}^{i}}

\newcommand{\uit}{\va{u}_{t}^{i}}

\newcommand{\wit}{\va{w}_{t}^{i}}

\newcommand{\zit}{\va{z}_{t}^{i}}
\newcommand{\zipt}{\va{z}_{t}^{i+1}}

\newcommand{\yit}{\va{y}_{t}^{i}}

\newcommand{\DP}{DP}
\newcommand{\SDDPd}{$\mathrm{SDDP}_{\!\mathrm{d}}$}
\newcommand{\SDDPc}{$\mathrm{SDDP}_{\!\mathrm{c}}$}
\newcommand{\DADP}{DADP}

\newcommand{\hesp}{\hspace{0.0cm}$\:$}
\newcommand{\hespa}{\hspace{0.5cm}$\;$}

\newcommand{\rouge}[1]{\textcolor{red}{#1}}
\newcommand{\bleue}[1]{\textcolor{blue}{#1}}

\newcommand{\cF}{\mathcal{F}}

\biboptions{authoryear}

\begin{document}

\begin{frontmatter}

\title{Stochastic decomposition applied to \\ large-scale hydro valleys management}

\author[ENSTAaddress]{P. Carpentier\corref{corresauthor}}
\cortext[corresauthor]{Corresponding author}
\ead{pierre.carpentier@ensta-paristech.fr}

\author[ENPCaddress]{J.-P. Chancelier}
\ead{jpc@cermics.enpc.fr}

\author[ENPCaddress]{V. Lecl\`{e}re}
\ead{vincent.leclere@enpc.fr}

\author[Efficacityaddress,ENPCaddress]{F. Pacaud}
\ead{f.pacaud@efficacity.com}

\address[ENSTAaddress]{UMA, ENSTA ParisTech, Universit\'{e} Paris-Saclay,
                       828 bd des Mar\'{e}chaux,91762 Palaiseau cedex France}
\address[ENPCaddress]{Université Paris-Est, CERMICS (ENPC), F-77455 Marne-la-Vallée, FRANCE}
\address[Efficacityaddress]{Efficacity, 14-20 boulevard Newton, F-77455 Marne-la-Vallée, FRANCE}

\date{\today}

\begin{abstract}
We are interested in optimally controlling a discrete time dynamical system
that can be influenced by exogenous uncertainties. This is generally called a
Stochastic Optimal Control~(SOC) problem and the Dynamic Programming~(DP)
principle is one of the standard way of solving it. Unfortunately, DP faces
the so-called curse of dimensionality: the complexity of solving DP equations
grows exponentially with the dimension of the variable that is
sufficient to take optimal decisions~(the so-called state variable).
For a large class of SOC problems, which includes important practical
applications in energy management, we propose an original way of obtaining
near optimal controls. The algorithm we introduce is based on Lagrangian
relaxation, of which the application to decomposition is well-known in the
deterministic framework. However, its application to such closed-loop problems
is not straightforward and an additional statistical approximation concerning
the dual process is needed. The resulting methodology is called Dual Approximate
Dynamic Programming (DADP). We briefly present DADP, give interpretations and
enlighten the error induced by the approximation. The paper is mainly devoted
to applying DADP to the management of large hydro valleys. The modeling of such
systems is presented, as well as the practical implementation of the methodology.
Numerical results are provided on several valleys, and we compare our approach
with the state of the art SDDP method.
\end{abstract}

\begin{keyword}
Discrete time stochastic optimal control \sep Decomposition methods \sep
Dynamic programming \sep Energy management
\MSC[2010] 90-08 \sep 90C06 \sep 90C15 \sep 90C39 \sep 49M27
\end{keyword}

\end{frontmatter}

%\linenumbers

%!TEX root = Article_Barrages.tex
%%%%%%%%%%%%%%%%%%%%%%%%%%%%%%%%%%%%%%%%%%%%%%%%%%%%%%%%%%%%%%%%%%%%%%
%                                                                    %
% 1. Introduction                                                    %
%                                                                    %
%%%%%%%%%%%%%%%%%%%%%%%%%%%%%%%%%%%%%%%%%%%%%%%%%%%%%%%%%%%%%%%%%%%%%%

\section{Introduction\label{sec:introduction}}

\subsection{Large-scale systems and energy applications\label{sec:intro-general}}

Consider a controlled dynamical system over a discrete and finite time horizon.
This system may be influenced by exogenous noises that affect its behavior.
We suppose that, at every instant, the decision maker is able to observe
the noises and to keep these observations in memory. Since it is generally
profitable to take available observations into account when designing future
decisions, we are looking for strategies rather than simple decisions.
Such strategies~(or policies) are feedback functions that map every instant
and every possible history of the system to a decision to be made.

\begin{figure}[ht]
\begin{center}
	\includegraphics[width=0.8 \textwidth]{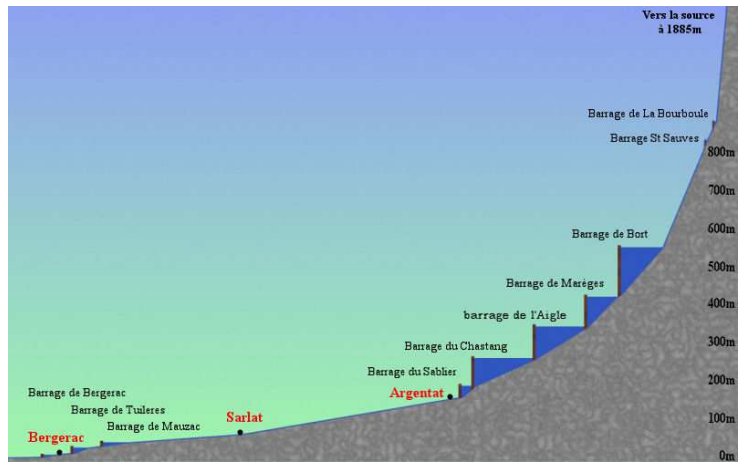}
\end{center}
\caption{\label{fig:Dordogne} The Dordogne river.}

\end{figure}

The typical applications we have in mind are in the field of energy management.
Consider a power producer that owns a certain number of power units.
Each unit has its own local characteristics such as physical constraints
that restrain the set of feasible decisions, and induces a production cost
or a revenue. The power producer has to control the power units so that
an overall goal is met. A classical example is the so-called unit commitment
problem \cite{takriti1996stochastic}
where the producer has to satisfy a global power demand at every instant.
The power demand, as well as other parameters such as unit breakdowns,
are random. The producer is looking for strategies that make the overall
production cost minimal, over a given time horizon. Another application,
which is considered in this paper, is the management of a large-scale
hydro valley: here the power producer manages a cascade of dams, and
maximizes the revenue obtained by selling the energy produced by
turbinating the water inside the dams. Both natural inflows in water
reservoirs and prices are random. In all these problems, both the number
of power units and the number of time steps are usually large \cite{deMatos_JCAM_2015}.
\subsection{Standard resolution methods\label{sec:intro-resolution}}

One classical approach when dealing with stochastic dynamic optimization
problems is to discretize the random inputs of the problem using a scenario tree.
Such an approach has been widely studied within the stochastic programming
community (\cite{HeitschRomisch09}, \cite{ShapiroDentchevaRuszczynski09}),
and also used to model and solve energy problems~\cite{PflugPichler14}.
One of the advantages of such a technique is that, as soon as the scenario
tree is drawn, the derived problem can be treated by classical mathematical
programming techniques. Thus, a number of decomposition methodologies have
been proposed (see for instance \cite{RockafellarWets91},
\cite{CarpentierCohenCulioliRenaud96}, \cite{ruszczynski1997decomposition}, \cite[Chap.~3]{StochasticProgramming03})
and applied to energy planning problems~\cite{BacaudLemarechalRenaudSagastizabal01}.
The way to combine the discretization of expectation together with
the discretization of information in a general setting has been presented in
\cite{HeitschRomischStrugarek05}, \cite{CCCD2015} and \cite{PflugPichler14}).
However, in a multi-stage setting, this methodology suffers from the drawback
that arises with scenario trees: as it was pointed out by~\cite{Shapiro06},
the number of scenarios needed to achieve a given accuracy grows exponentially
with the number of time steps of the problem.

The other natural approach to solve SOC problems is to rely on the dynamic
programming~(DP) principle (see \cite{Bellman57}, \cite{Puterman94}, \cite{BertsekasDP}).
The core of the DP
approach is the definition of a state variable that is, roughly speaking,
the variable that, in conjunction with the time variable, is sufficient
to take an optimal decision at every instant. It does not have the drawback
of the scenario trees concerning the number of time steps since strategies
are, in this context, depending on a state variable whose space dimension
does not grow with time\footnote{In the case of power management,
the state dimension is usually linked to the number of power units.}.
However, DP suffers from another drawback which is the so-called
\emph{curse of dimensionality}: the complexity of solving the DP equation
grows exponentially with the state space dimension. Hence, solving the DP
equation by brute force is generally intractable when the state space
dimension goes beyond several units.
%Recently,\vl{Recently for 2009 paper might not be up-to-date}
In \cite{VezolleVialleWarin09}, the authors
were able to solve DP on a~$10$ state variables energy management problem,
using parallel computation coupled with adequate data distribution,
but the DP limits are around $5$ state variables in a straightforward
use of the method.

Another popular idea is to represent the value functions~(solutions of the
DP equation) as a linear combination of a priori chosen basis functions
(see \cite{BellmanDreyfus59} \cite{BertsekasTsitsiklis96}). This approach,
called Approximate Dynamic Programming (ADP) has become very popular and
the reader is referred to \cite{Powell11} and \cite{Bertsekas12}
for a precise description of ADP. This approximation drastically reduces
the complexity of solving the DP equation. However, in order to be
practically efficient, such an approach requires some a priori information
about the problem, in order to define a well suited functional subspace.
Indeed, there is no systematic means to choose the basis functions and several
choices have been proposed in the literature \cite{TsitsiklisVanRoy96}.

Last but not least is the popular
DP-based method called Stochastic Dual Dynamic Programming (SDDP).
Starting with the seminal work \cite{VanSlyke_SIAM_1969}, the SDDP method
has been designed in~\cite{PereiraPinto91}. It has been widely used in
the energy management context and lately regained interest in the Stochastic Programming
community~\cite{PhilpottGuan08}, \cite{Shapiro10}. The idea is somehow
to extend Kelley's cutting plane method to the case of multi-stage problems.
Alternatively it can be seen as a multistage Benders (or L-shaped) decomposition method
with sampling.
It consists in a succession of forward (simulation) and backward (Bellman
function refining) passes that ultimately aim at approaching the Bellman
function by the supremum of affine hyperplanes (cuts) generated during the backward passes.
It provides an efficient alternative to simply discretizing the state space
to solve the DP equation. In the convex case with finite support random
variables, it has been proven by~\cite{Girardeau_MOR_2015} that the method
converges to optimality.

\subsection{Decomposition approach\label{sec:intro-decomposition}}

When dealing with large-scale optimization problems,
the decomposition-coordination approach aims at finding a solution to
the original problem by iteratively solving subproblems of smaller dimension.
In the deterministic case, several types of decomposition have been proposed
(e.g. by prices, by quantities or by interaction prediction) and unified
in~\cite{Cohen80} using a general framework called Auxiliary Problem Principle.
In the open-loop stochastic case, i.e. when controls do not rely on any
observation, it is proposed in~\cite{CohenCulioli90} to take advantage of both
decomposition techniques and stochastic gradient algorithms.

The natural extension of these techniques in the closed-loop stochastic case
(see~\cite{BartyRoyStrugarek05}) fails to provide decomposed state dependent
strategies. Indeed, the optimal strategy of a subproblem depends on the state
of the whole system, and not only on the local state. In other words,
decomposition approaches are meant to decompose the control space, namely
the range of the strategy, but the numerical complexity of the problems
also arises because of the dimensionality of the state space, that is
to say the domain of the strategy.

We recently proposed a way to use price decomposition within the closed-loop
stochastic case. The coupling constraints, namely the constraints preventing
the problem from being naturally decomposed, are dualized using a Lagrange
multiplier (price). At each iteration, the price decomposition algorithm
solves each subproblem using the current price, and then uses the solutions
to update it. In the stochastic context, the price is a random process
whose dynamics is not available, so the subproblems do not in general fall
into the Markovian setting. However, in a specific instance of this problem
\cite{TheseStrugarek}, the author exhibited a dynamics for the optimal multiplier
and showed that these dynamics were independent with respect to the decision
variables. Hence it was possible to come down to the Markovian framework and
to use DP to solve the subproblems in this case. Following this idea, it is
proposed in~\cite{BartyCarpentierGirardeau09} to choose a parameterized dynamics
for these multipliers in such a way that solving subproblems using DP becomes
possible. While the approach, called Dual Approximate Dynamic Programming~(DADP),
showed promising results on numerical examples, it suffers from the fact that
the induced restrained dual space is non-convex, leading to some numerical
instabilities. Moreover, it was not possible to give convergence results
for the algorithm.

The method has then been improved through a series of PhD theses
(\cite{TheseGirardeau}, \cite{TheseAlais} and \cite{TheseLeclere}) both
from the theoretical and from the practical point of view. The methodology
still relies on Lagrangian decomposition. In order to make the resolution
of subproblems tractable, the core idea is to replace the current Lagrange
multiplier by its conditional expectation with respect to some
\emph{information process}, at every iteration. This information process
has to be a priori chosen and adapted to the natural filtration. Assume that
the information process is driven by a dynamics: the state in each
subproblem then consists of the original state augmented by the new state
induced by the information process, making the resolution of the subproblem
tractable by DP. The quality of the results produced by
the algorithm highly depends on the choice of this information variable.
An interesting point is that approximating the multipliers
by their conditional expectations has
an interpretation in terms of a relaxed optimization problem: this revisited
DADP algorithm in fact aims at solving an approximate primal problem where
the almost-sure coupling constraint has been replaced by its
conditional expectation with respect to the information variable.
In other words, this methodology consists in solving a relaxed
primal problem, hence giving a lower bound of the true optimal cost.
Another consequence of this approximation is that the solutions
obtained by the DADP algorithm do not satisfy the initial almost-sure
coupling constraint, so we must rely on a heuristic procedure to produce
a feasible solution to the original problem.

\subsection{Contents of the paper\label{sec:intro-contents}}

The main contribution of the paper is to give a practical algorithm aiming
at solving large scale stochastic optimal control problems and providing
closed-loop strategies. The numerous approximations used in the algorithm,
and especially the one allowing for feasible strategies, make difficult to
theoretically assess the quality of the solution finally adopted.
Nevertheless, numerical implementation shows that the method is promising
to solve large scale optimization problems such as those encountered
in the field of energy management.

The paper is organized as follows. In \S\ref{sec:framework},
we present the hydro valley management problem, the corresponding general
SOC formulation and the DP principle. We then concentrate on the spatial
decomposition of such a problem and the difficulties of using DP
at the subproblem level. In \S\ref{sec:DADP}, we present the DADP
method and give different interpretations. We then propose a way to
recover an admissible solution from the DADP results and we briefly discuss
the theoretical and practical questions associated to the convergence and
 implementation of the method. Finally, in \S\ref{sec:numerical},
we apply the DADP method to the management of hydro valleys. Different
examples, corresponding to either academic or realistic valleys,
are described. A comparison of the method with SDDP is outlined.

\subsection{Notations}
We will use the following notations, considering a probability space $(\omeg,\trib,\prbt)$:
\begin{itemize}
	\item bold letters for random variables, normal font for their realizations;
	\item $\va X \preceq \cF_t$ (resp. $\va X \preceq \va Y$)
	means that the random variable $\va X$ is measurable with respect to the $\sigma$-algebra $\cF_t$
	(resp. with respect to the $\sigma$-algebra generated by $\va Y$, denoted
    by~$\sigma(\va{y})$);
	\item $x$ generally stands for the state, $u$ for the control, $w$ for an exogeneous noise;
	\item $L_t$ for a cost function at time $t$, $K$ for a final cost function;
	\item $V_t$ represent a Bellman's value function at time $t$;
	\item $\dc{i,j}$ is the set of integer between $i$ and $j$;
	\item the notation~$\va{x}^{i}$ (resp.~$\va{u}^{i}$ and $\va{z}^{i}$) stands for the discrete
time state process~$(\va{x}_{0}^{i},\ldots,\va{x}_{T}^{i})$ (resp. control processes
$(\va{u}_{0}^{i},\ldots,\va{u}_{T-1}^{i})$, $(\va{z}_{0}^{i},\ldots,\va{z}_{T-1}^{i})$).
\end{itemize}

%!TEX root = Article_Barrages.tex
%%%%%%%%%%%%%%%%%%%%%%%%%%%%%%%%%%%%%%%%%%%%%%%%%%%%%%%%%%%%%%%%%%%%%%
%                                                                    %
% 2. Mathematical formulation                                        %
%                                                                    %
%%%%%%%%%%%%%%%%%%%%%%%%%%%%%%%%%%%%%%%%%%%%%%%%%%%%%%%%%%%%%%%%%%%%%%

\section{Mathematical formulation\label{sec:framework}}

In this section, we present the modeling of a hydro valley
and the associated optimization framework.

\subsection{Dams management problem\label{sec:frame-problem}}

We consider a hydro valley constituted of $N$~cascaded dams as represented in
Figure~\ref{fig:Valley}. The water turbinated at a dam produces energy which is
sold on electricity markets, and then enters the nearest downstream dam.\footnote{Note
that the valley geometry may be more complicated than a pure cascade: see for example
the valleys represented at Figure~\ref{fig:academic}.}
\begin{figure}[ht]
\begin{center}
	\includegraphics[width=0.8 \textwidth]{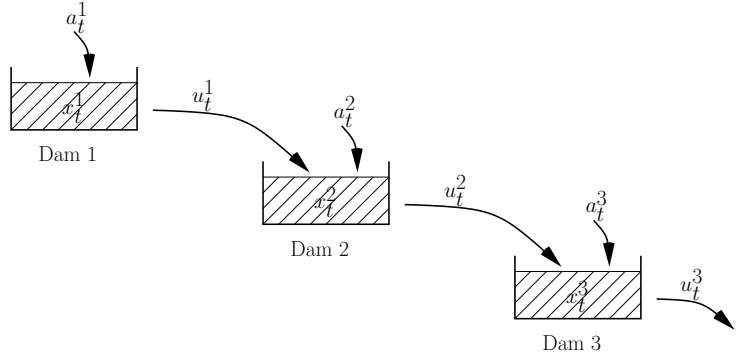}
\end{center}
\caption{\label{fig:Valley} Operating scheme of a hydro valley with $3$ dams.}
\end{figure}

We formulate the problem of maximizing the cascade revenue over a discrete time
horizon~$\{0,1,\ldots,T\}$. The representative variables of dam~$i$ at time~$t$ are:
$u_{t}^{i}$ for the water turbinated, $x_t^i$ for the current water volume, $a_t^i$ for the natural
water inflow entering dam $i$, $p_t^i$ for the market value of the water of dam $i$.
% SHORTER
% \begin{itemize}
% \item $u_{t}^{i}$: water turbinated at dam~$i$ at time~$t$,
% \item $x_{t}^{i}$: water volume in dam~$i$ at time~$t$,
% \item $a_{t}^{i}$: water natural inflow entering dam~$i$ at time~$t$,
% \item $p_{t}^{i}$: market price available for dam~$i$ at time~$t$.
% \end{itemize}
The randomness is given by $w_{t}^{i} = (a_{t}^{i},p_{t}^{i})$.
%and we denote by $w_{t}$ the vector $(w_{t}^{1},\ldots,w_{t}^{N})$.

The modeling of a dam takes into account a possible overflow: the spilled water
does not produce electricity, but enters the next dam (see~Figure~\ref{fig:Dam}).
\begin{figure}[ht]
\begin{center}
	\includegraphics[width=0.4 \textwidth]{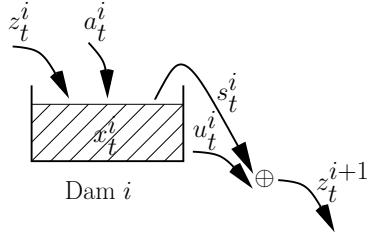}
\end{center}
\caption{\label{fig:Dam} Dam behavior.}
\end{figure}
\begin{itemize}
\item The dam dynamics reads
\begin{align*}
x_{t+1}^{i} & = x_{t}^{i}-u_{t}^{i}+a_{t}^{i}+z_{t}^{i}-s_{t}^{i}
              = f_{t}^{i}(x_{t}^{i},u_{t}^{i},w_{t}^{i},z_{t}^{i}) \eqfinv
\end{align*}
where~$s_{t}^{i}$ is the volume of water spilled by overflow of the dam:
\begin{equation*}
s_{t}^{i} = \max\ba{0,x_{t}^{i}-u_{t}^{i}+a_{t}^{i}+z_{t}^{i}-\overline{x}^{i}} \eqfinp
\end{equation*}
The constant value~$\overline{x}^{i}$ stands for the maximal capacity of dam~$i$.
The outflow of dam~$i$, that is, the sum of the turbinated
water and of the spilled water, is denoted by~$z_{t}^{i+1}$:
\begin{align*}
z_{t}^{i+1} & = u_{t}^{i} + s_{t}^{i}  = g_{t}^{i}(x_{t}^{i},u_{t}^{i},w_{t}^{i},z_{t}^{i}) \eqfinp
\end{align*}
Note that these dynamic equations are nonlinear because of the~$\max$ operator
in the definition of the spilled water volume.
We assume the \emph{Hazard-Decision} information structure: the control~$u_{t}^{i}$
applied at time~$t$ is chosen once the noise~$w_{t}^{i}$ at time~$t$ has been
observed. It is thus possible to ensure that the dam always remains above its
minimal admissible volume~$\underline{x}^{i}$ by limiting the control range:
\begin{equation*}
\underline{u}^{i} \leq u_{t}^{i} \leq \min\ba{\overline{u}^{i},x_{t}^{i}+a_{t}^{i}+z_{t}^{i}-\underline{x}^{i}} \eqfinp
\end{equation*}
\begin{remark}
As will be seen in~\S\ref{sec:numerical}, the typical time step length
we use is the month (with a time horizon of one year). It is thus reasonable
to assume the hazard-decision framework, the control applied for a given month
being in fact implemented each day taking into account the observed information
on a daily basis.
\end{remark}
 \item The gain of the dam is the sum of different terms.
 \begin{itemize}
 \item Gain at each time $t < T$:
 \begin{equation*}
 p_{t}^{i}u_{t}^{i} - \epsilon (u_{t}^{i})^{2} \; ;
 \end{equation*}
 we denote by~$L_{t}^{i}(x_{t}^{i},u_{t}^{i},w_{t}^{i},z_{t}^{i})$
 the \emph{opposite} of this gain (cost).
 The quadratic terms~$\epsilon (u_{t}^{i})^{2}$ ensure the strong convexity
 of the cost function. These terms model the operating cost of the turbine
 and are usually small.
 \item Final gain at time $T$:
 \begin{equation*}
 - a^{i} \min\na{0,x_{T}^{i}-\widehat{x}^{i}}^{2} \; ;
 \end{equation*}
 we again denote by~$K^{i}\bp{x_{T}^{i}}$ the \emph{opposite} of this gain.
 It corresponds to a quadratic penalization around a target value~$\widehat{x}^{i}$
 representing the desired water volume in the dam at the end of the time horizon.
 \end{itemize}
 Taking into account the opposite of the gains, we thus have to deal with
 a \emph{minimization problem}.
%\vl{J'ai pass\'{e} en min (ancienne version comment\'{e}e)}
%\item The objective function of dam $i$ is the sum of different terms.
%\begin{itemize}
%\item Cost at each time $t < T$:
%\begin{equation*}
%L_{t}^{i}(x_{t}^{i},u_{t}^{i},w_{t}^{i},z_{t}^{i}) = -p_{t}^{i}u_{t}^{i} + \epsilon (u_{t}^{i})^{2} \; ;
%\end{equation*}
%where the quadratic terms~$\epsilon (u_{t}^{i})^{2}$ ensure the strong convexity
%of the cost function. These terms model the operating cost of the turbine
%and are usually small.
%\item Final cost at time $T$:
%\begin{equation*}
%K^{i}\bp{x_{T}^{i}} = a^{i} \min\na{0,x_{T}^{i} + \widehat{x}^{i}}^{2} \; .
%\end{equation*}
%It corresponds to a quadratic penalization around a target value~$\widehat{x}^{i}$
%representing the desired water volume in the dam at the end of the time horizon.
%\end{itemize}
\end{itemize}
% SHORTER
% As already mentioned, we assume the Hazard-Decision information structure
% with full observation of the past noises. The controls at time~$t$ thus
% depend on the observations of the random disturbances up to time~$t$.
% Then all the variables become random in the formulation of the problem.

Thus the global optimization problem reads
\begin{subequations}
\label{pb:normal}
\begin{align}
\min_{(\va{x},\va{u},\va{z})} \quad &
\bgesp{ \sum_{i=1}^{N} \Bp{
\sum_{t=0}^{T-1} L_{t}^{i}\bp{\xit,\uit,\wit,\zit} + K^{i}\bp{\xitf}}} \eqfinv
\label{pb:normal-cou} \\
s.t. \qquad & \va{x}^{i}_{0} \quad \text{given} \eqfinv \nonumber \\
& \xitp = f_{t}^{i}(\xit,\uit,\wit,\zit)
          \eqsepv i\in\dc{1,N} \eqsepv t\in\dc{0,T-1} \eqfinv \label{pb:normal-dyn} \\
& \uit \preceq \sigma(\va{w}_{0},\ldots,\va{w}_{t}) \eqsepv
               \qquad\:\: i\in\dc{1,N} \eqsepv t\in\dc{0,T-1} \label{pb:normal-mes} \eqfinv \\
& \va{z}^{1}_{t} = 0 \eqfinv \nonumber \\
& \zipt = g_{t}^{i}(\xit,\uit,\wit,\zit)
          \eqsepv i\in\dc{1,N} \eqsepv t\in\dc{0,T-1} \eqfinp \label{pb:normal-cpl}
\end{align}

Equations~\eqref{pb:normal-mes} represent the so-called non-anticipativity constraints,
that is, the fact that each control~$\uit$, considered as a random variable, has to be
\emph{measurable} with respect to the sigma-field $\sigma(\va{w}_{0},\ldots,\va{w}_{t})$
generated by the noise sequence~$(\va{w}_{0},\ldots,\va{w}_{t})$ up to time~$t$.
\end{subequations}

\subsection{A more generic formulation\label{sec:frame-compact}}

%\vlil{pour la version ejor je propose de passer ce paragraphe en premier et de donner ensuite la def de chaque fonction pour le cas barrages.}

With a slight abuse of notation,\footnote{which consists in denoting by~$\uit$
the pair~$(\uit,\zit)$} the stochastic optimization problem formulated
at \S\ref{sec:frame-problem} reads
\begin{subequations}
\label{pb:compact}
\begin{align}
\label{pb:compact-cou}
\min_{(\va{x}^{i},\va{u}^{i})_{i\in\dc{1,N}}} \;
& \bgesp{\sum_{i=1}^{N}
  \Bp{\sum_{t=0}^{T-1}\coutint^{i}_{t}(\va{x}^{i}_{t},\va{u}^{i}_{t},\va{w}_{t})+
      \coutfin^{i}(\va{x}^{i}_{T})}} \eqfinv
      \\
% \intertext{subject to dynamics constraints:}
\text{s.t.} \quad
& \va{x}^{i}_{t+1} = \dynamics^{i}_{t}(\va{x}^{i}_{t},\va{u}^{i}_{t},\va{w}_{t})
  \eqfinv \quad \va{x}^{i}_{0} \quad \text{given}
  \label{pb:compact-dyn} \eqfinv
  \\
% \intertext{to measurability constraints:}
& \va{u}^{i}_{t} \preceq \sigma(\va{w}_{0},\ldots,\va{w}_{t})
  \label{pb:compact-mes} \eqfinv
  \\
% \intertext{and to dams interaction constraints:}
& \sum_{i=1}^{N}\production^{i}_{t}(\va{x}^{i}_{t},\va{u}^{i}_{t},\va{w}_{t})=0
  \label{pb:compact-cpl} \eqfinp
\end{align}
Constraints~\eqref{pb:compact-dyn} represent the dynamics and
constraints~\eqref{pb:compact-mes} are the non-anticipativity constraints.
The last constraints~\eqref{pb:compact-cpl} express the interactions between
the dams in a more general way than Equations~\eqref{pb:normal-cpl}. They
represent an \emph{additive} coupling with respect to the different production
units, which is termed the ``spatial coupling of the problem''. Such a general
modeling covers other cases than the cascade problem, such that the unit
commitment problem, or the problem of exchanging energy on a smart grid.
\end{subequations}

\subsection{Dynamic Programming like approaches\label{sec:frame-DP}}

In the remainder of the paper, we assume that we are in the so-called
\emph{white noise} setting.

\begin{assumption}
Noises $\va{w}_{0},\ldots,\va{w}_{T-1}$ are independent over time.
\label{ass:white}
\end{assumption}
This assumption is of paramount importance in order to use
Dynamic Programming or related approaches such that Stochastic
Dual Dynamic Programming since in that case the controls given
by DP or SDDP are the \emph{optimal} ones for Problem~\ref{pb:compact}
(they are given as feedback functions depending on the state variable).
This assumption can be alleviated, in the case where it is possible
to identify a dynamics in the noise process (such as an ARMA model),
and by incorporating this new dynamics in the state
variables (see e.g. \cite{Maceira04} on this topic).

Under Assumption~\ref{ass:white}, Dynamic Programming (DP) applies
to Problem~\eqref{pb:compact}
(see e.g ~\cite{Bellman57,Puterman94,BertsekasDP} for the general theory):
there is no optimality loss to seek each control $\va{u}^{i}_{t}$ at time $t$
as a function of both the state and the noise at time~$t$.\footnote{Remember that
we have assumed the Hazard-Decision framework.}  Then, thanks to the measurability
properties of the control the Bellman functions~$V_{t}$
are obtained by solving the Dynamic Programming equation backwards in time:
\begin{subequations}
\label{eq:Bellman-SDP}
\begin{align}
V_{T}(x_T)
& = \sum_{i=1}^{N} \coutfin^{i}(x^{i}_T) \eqfinv \\
V_{t}(x_t)
& =
    \espe \bigg(\min_{{u}^{1},\ldots,{u}^{N}}
    \sum_{i=1}^{N} \coutint^{i}_{t}(x^{i},\va{u}^{i},\va{w}_{t})
   +  V_{t+1}\bp{\dynamics_{t}(x_t,{u}_t,\va{w}_{t}))}
  \bigg) \eqfinp
\end{align}
\end{subequations}
where $x_t=(x_t^1,\dots,x_t^N)$, $u_t=(u_t^1,\dots,u_t^N)$ and
$f_t(x_t,u_t,\va w_t)$ is the collection of new states
$f^i_t(x^1_t,u^i_t,\va w_t)$.

% Thanks to the measurability properties of the control, the min and expectation operators can be exchanged
% (see \cite[Theorem~14.60]{RockWets}) so that the minimization can be
% performed~$\omega$ by~$\omega$:
% \begin{multline}
% \label{eq:Bellman-SDP}
% V_{t}(x^{1},\ldots,x^{N}) = \espe \bigg( \min_{u^{1},\ldots,u^{N}} \;
% \sum_{i=1}^{N} \coutint^{i}_{t}(x^{i},u^{i},\va{w}_{t}) \bigg. \\
% \bigg. +  V_{t+1}\bp{\dynamics^{1}_{t}(x^{1},u^{1},\va{w}_{t}),\ldots,
%                      \dynamics^{N}_{t}(x^{N},u^{N},\va{w}_{t})} \bigg) \eqfinp
% \end{multline}
The DP equation is agnostic to whether the state and control variables are
continuous or discrete, whether the constraints and the cost functions
are convex or not, etc.
However, in order to exhaustively solve it we need to have discrete state,
and  to be able to solve each equation to optimality. In practice, the method
is subject to the curse of dimensionality and cannot be used
for large-scale optimization problems. For example, applying DP to
dams management problems is practically untractable for more than five dams
(see the results given at~\S\ref{sec:num-academic}).

Another way to compute the Bellman functions asociated to Problem~\eqref{pb:compact}
is to use the Stochastic Dual Dynamic Programming (SDDP) method.
The method has been first described in~\cite{PereiraPinto91}, and its
convergence has been analysed in~\cite{PhilpottGuan08} for the linear
case and in~\cite{Girardeau_MOR_2015} for the general convex case.
SDDP recursively constructs an approximation of each Bellman function
as the supremum of a number of affine functions, thus exploiting
the convexity of the Bellman functions (arising from the convexity
of the cost and constraint functions).
SDDP has been used for a long time for solving large-scale hydrothermal
problems (see~\cite{deMatos_JCAM_2015} and the references therein) and
allows to push the limits of DP in terms of state dimension
(see the results in~\S\ref{sec:num-challenge}).

\subsection{Spatial coupling and approach by duality\label{sec:frame-coupling}}

A standard way to tackle large-scale optimization problems is to use
Lagrange relaxation in order to split the original problem into a collection
of smaller subproblems by dualizing coupling constraints.
As far as Problem~\eqref{pb:compact} is concerned, we have in mind to
use DP for solving the subproblems and thus want to dualize
the spatial coupling constraints~\eqref{pb:compact-cpl} in order
to formulate subproblems, each incorporating a single dam.
The associated Lagrangian~$\mathcal{L}$ is accordingly
\begin{multline*}
\mathcal{L}\bp{\va{x},\va{u},\va{\Lambda}} = \Bgesp{\sum_{i=1}^{N}\bgp{\sum_{t=0}^{T-1} \coutint^{i}_{t}(\va{x}^{i}_{t},\va{u}^{i}_{t},\va{w}_{t}) + \coutfin^{i}(\va{x}^{i}_{T}) \\
+ \sum_{t=0}^{T-1} \va{\Lambda}_{t} \cdot \production^{i}_{t}(\va{x}^{i}_{t},\va{u}^{i}_{t},\va{w}_{t})}}
\eqfinv
\end{multline*}
where the multiplier~$\va{\Lambda}_{t}$ associated to
Constraint~\eqref{pb:compact-cpl} is a random variable. From
the measurability of the variables~$\va{x}^{i}_{t}$, $\va{u}^{i}_{t}$
and~$\va{w}_{t}$, we can assume without loss of optimality that
the multipliers~$\va{\Lambda}_{t}$ are $\sigma(\va{w}_{0},\ldots,\va{w}_{t})$-measurable
random variables.

In order to be able to apply duality theory to the problem (which is mandatory
for algorithmic resolution), we make the two following assumptions.
\begin{assumption}
A saddle point of the Lagrangian~$\mathcal{L}$ exists.
\label{ass:cqc}
\end{assumption}
\begin{assumption}
Uzawa algorithm applies to compute a saddle-point of~$\mathcal{L}$.
\label{ass:uzawa}
\end{assumption}
Assumption~\ref{ass:cqc} corresponds to a Constraint Qualification condition
and ensures the existence of an optimal multiplier. Assumption~\ref{ass:uzawa}
allows to use a gradient ascent algorithm to compute the optimal multiplier.
An important question in order to be able to satisfy these two assumptions
is the choice of the spaces where the various random variables of the problem
are living in. Duality theory and associated algorithms have been extensively
studied in the framework of Hilbert spaces \cite{EkelandTemam92}, but
the transition to the framework of stochastic optimal control poses difficult
challenges \cite{Rockafellar_PJM_1968,Rockafellar_PJM_1971}, which will be briefly
presented in \S\ref{sec:DADP-questions}.\footnote{Remember that the aim of the present
paper is mainly to present numerical results. The reader is referred
to \cite{TheseLeclere} for these theoretical questions.} One way to get
rid of these difficulties is to assume that the space~$\omeg$ is finite.

When using the Uzawa algorithm to compute a saddle-point of the Lagrangian,
the minimization step with respect to~$(\va{x}^{i},\va{u}^{i})_{i\in\dc{1,N}}$
splits in~$N$ independent subproblems each depending on a single pair
$(\va{x}^{i},\va{u}^{i})$, and therefore allows for a dam by dam decomposition.
More precisely, the~$k$-th iteration of Uzawa algorithm consists of the two
following steps.
\begin{enumerate}
\item Solve Subproblem $i$, $i\in\dc{1,N}$,
with fixed $\va{\Lambda}^{(k)}$:
\begin{subequations}
\label{pb:subproblem}
% \begin{multline*}
% \label{pb:subproblem-cou}

% \end{multline*}
% subject to
\begin{align}
%\label{pb:subproblem-cou}
\nonumber
\min_{\va{x}^{i},\va{u}^{i}} \; &
\Besp{\sum_{t=0}^{T-1}
\coutint^{i}_{t}(\va{x}^{i}_{t},\va{u}^{i}_{t},\va{w}_{t})  + %\\
 \va{\Lambda}_{t}^{(k)} \cdot\production^{i}_{t}(\va{x}^{i}_{t},\va{u}^{i}_{t},\va{w}_{t})
+ \coutfin^{i}(\va{x}^{i}_{T})}
                 \eqfinv \\
s.t. \quad &  \va{x}^{i}_{t+1} = \dynamics^{i}_{t}(\va{x}^{i}_{t},\va{u}^{i}_{t},\va{w}_{t})
  \eqfinv \va{x}^{i}_{0} \quad \text{given} \label{pb:subproblem-dyn} \\
& \va{u}^{i}_{t} \;\;\: \preceq \sigma(\va{w}_{0},\ldots,\va{w}_{t})
  \label{pb:subproblem-mes} \eqfinv
\end{align}
whose solution is denoted $\bp{\va{u}^{i,(k)},\va{x}^{i,(k)}}$.
\end{subequations}
\item Use a gradient step to update the multipliers $\va{\Lambda}_{t}$:
\begin{equation}
\label{eq:multiplier-update}
\va{\Lambda}_{t}^{(k+1)} = \va{\Lambda}_{t}^{(k)} + \rho_t
\bgp{\sum_{i=1}^{N}\production_{t}^{i}\bp{\va{x}_t^{i,(k)},\va{u}_t^{i,(k)},\va{w}_{t}}}
\eqfinp
\end{equation}
\end{enumerate}

Consider the resolution of Subproblem~\eqref{pb:subproblem}. This subproblem
only involves the ``physical'' state variable~$\va{x}^{i}_{t}$ and the control variable~$\va{u}^{i}_{t}$, a situation which seems favorable to DP. It also
involves two exogeneous random processes, namely $\va{w}$ and $\va{\Lambda}^{(k)}$.
The white noise Assumption~\ref{ass:white} applies for the first process~$\va{w}$,
but not for the second one~$\va{\Lambda}^{(k)}$, so that the state of the system
cannot be summarized
by the physical state~$\va{x}^{i}_{t}$~! Moreover if we just use the fact
that~$\va{\Lambda}^{(k)}_{t}$ is measurable with respect to the past
noises, the state of the system must incorporate all noises prior to
time~$t$, that is, $(\va{w}_{0},\ldots,\va{w}_{t})$. The state size
of the subproblem increases with time. Without some additional knowledge
on the process~$\va{\Lambda}^{(k)}$, DP cannot be
applied in a straightforward manner: something has to be compressed
in order to use Dynamic Programming.

%!TEX root = Article_Barrages.tex
%%%%%%%%%%%%%%%%%%%%%%%%%%%%%%%%%%%%%%%%%%%%%%%%%%%%%%%%%%%%%%%%%%%%%%
%                                                                    %
% 3. Dual Approximate Dynamic Programming                            %
%                                                                    %
%%%%%%%%%%%%%%%%%%%%%%%%%%%%%%%%%%%%%%%%%%%%%%%%%%%%%%%%%%%%%%%%%%%%%%

\section{Dual Approximate Dynamic Programming\label{sec:DADP}}

For a very specific instance of Problem~\eqref{pb:compact}, \cite{TheseStrugarek}
exhibited the dynamics of the optimal multiplier. Hence it was possible to come down
to the Markovian framework and to use DP to solve the subproblems~\eqref{pb:subproblem}
with an augmented space, namely the ``physical" state~$\va{x}^{i}_{t}$ and the state
associated to the mutiplier's dynamics. Following this idea for a general
Problem~\eqref{pb:compact}, \cite{BartyCarpentierGirardeau09} proposed to choose
a parameterized dynamics for the multiplier: then solving the subproblems using DP
becomes possible, the parameters defining the multiplier dynamics being updated
 at each iteration of the Uzawa algorithm. This new approach, called Dual Approximate
Dynamic Programming~(DADP), has then improved through a series of PhD theses
(\cite{TheseGirardeau}, \cite{TheseAlais} and \cite{TheseLeclere}) both from
the theoretical and from the practical point of view. We give here a brief overview
of the current DADP method.

\subsection{DADP core idea and associated algorithm\label{sec:DADP-core}}

In order to overcome the obstacle explained at~\S\ref{sec:frame-coupling}
concerning the random variables~$\va{\Lambda}_{t}^{(k)}$, we choose
a random variable~$\va{y}_{t}$ at each time~$t$,\footnote{Note that
the information variables~$\va{y}_{t}$ may depend on the subproblem
index~$i$: see~\cite{TheseGirardeau} for further details.}
each~$\va{y}_{t}$ being measurable with respect to the noises up
to time~$t$ $\bp{\va{w}_{0},\ldots,\va{w}_{t}}$.
We call $\va{y}=\bp{\va{y}_{0},\ldots,\va{y}_{T-1}}$
the \emph{information process} associated to Problem~\eqref{pb:compact}.

\subsubsection{Method foundation\label{sec:DADP-core-foundation}}

The core idea of DADP is to replace the multiplier $\va{\Lambda}_{t}^{(k)}$
by its conditional expectation $\nespc{\va{\Lambda}_{t}^{(k)}}{\va{y}_{t}}$
with respect to~$\va{y}_{t}$.
From an intuitive point of view, the resulting optimization problem will be
a good approximation of the original one if~$\va{y}_{t}$ is close
to the random variable $\va{\Lambda}_{t}^{(k)}$. Note that we require that
the information process is not influenced by controls because introducing
a dependency of the conditioning term with respect to the control would
lead to very serious difficulties for optimization.

Using this core idea, we replace Subproblem~\eqref{pb:subproblem} by:
\begin{subequations}
\label{pb:subproblem-approx}
% \begin{multline}

% \end{multline}
% subject to
\begin{align}
\label{pb:subproblem-cou-approx}
\min_{\va{x}^{i},\va{u}^{i}} \qquad
& \bgesp{\sum_{t=0}^{T-1}
  \Bp{\coutint^{i}_{t}(\va{x}^{i}_{t},\va{u}^{i}_{t},\va{w}_{t})
  + \coutfin^{i}(\va{x}^{i}_{T})}
  \nonumber \\
& \qquad\qquad\qquad\qquad
  + \nespc{\va{\Lambda}_{t}^{(k)}}{\va{y}_{t}} \cdot
  \production^{i}_{t}(\va{x}^{i}_{t},\va{u}^{i}_{t},\va{w}_{t})}
  \eqfinv \\
\text{s.t.} \quad
& \va{x}^{i}_{t+1} = \dynamics^{i}_{t}(\va{x}^{i}_{t},\va{u}^{i}_{t},\va{w}_{t})
  \eqsepv \quad \va{x}^{i}_{0} \quad \text{given}
  \eqfinv \label{pb:subproblem-dyn-approx} \\
& \va{u}^{i}_{t} \;\;\: \preceq \sigma(\va{w}_{0},\ldots,\va{w}_{t})
  \label{pb:subproblem-mes-approx} \eqfinp
\end{align}
\end{subequations}
According to the Doob property \cite[Chapter~1, p.~18]{DellacherieMeyer},
the~$\va{y}_{t}$-measurable random variable~$\nespc{\va{\Lambda}_{t}^{(k)}}{\va{y}_{t}}$
can be represented by a measurable mapping~$\mu_{t}^{(k)}$, that is,
\begin{equation}
\mu_{t}^{(k)}(\va{y}_{t}) = \bespc{\va{\Lambda}_{t}^{(k)}}{\va{y}_{t}} \eqfinv
\label{eq:Doob}
\end{equation}
so that Subproblem~\eqref{pb:subproblem-approx} in fact involves the two
fixed random processes $\va{w}$ and $\va{y}$. If the process~$\va{y}$
follows a non-controlled Markovian dynamics driven by the noise process~$\va{w}$, i.e. if there exists functions $h_t$ such that
\begin{equation}
\va{y}_{t+1} = h_{t} (\va{y}_{t},\va{w}_{t}) \eqfinv
\label{eq:dyn-info}
\end{equation}
then~$(\va{x}^{i}_{t},\va{y}_{t})$ is a valid state for the subproblem
and DP applies.

\subsubsection{DADP algorithm\label{sec:DADP-core-algorithm}}

Assume that the information process~$\va{y}$ follows the dynamics~\eqref{eq:dyn-info}.
\begin{itemize}
\item The first step of the DADP algorithm at iteration~$k$ consists in solving
the subproblems~\eqref{pb:subproblem-approx} with~$\va{\Lambda}_{t}^{(k)}$ fixed,
that is, with~$\mu_{t}^{(k)}(\cdot)$ given. It is done by solving the Bellman
functions associated to each subproblem~$i$, that is,
\begin{align*}
V_{T}^{i}(x^{i},y) =
& \; \coutfin^{i}(x) \eqfinv \\
V_{t}^{i}(x^{i},y) =
& \; \besp{Q_t^k (x^i,y,\va w_t)}
% & \; \Besp{ \min_{u^{i}}
%         \coutint^{i}_{t}(x^{i},u^{i},\va{w}_{t}) \\
% & \hspace{2.0cm} + \mu_{t}^{(k)}(y)
%                    \cdot \production^{i}_{t}(x^{i},u^{i},\va{w}_{t})
%                  + V_{t+1}^{i}\bp{\va{x}_{t+1}^{i},\va{y}_{t+1}^{i}}
%    } \eqfinv \\
% \intertext{subject to the dynamics:}
% & \va{x}_{t+1}^{i} = \dynamics^{i}_{t}(x^{i},u^{i},\va{w}_{t}) \eqfinv \\
% & \va{y}_{t+1} = h_{t}(y,\va{w}_{t}) \eqfinp
\end{align*}
where $Q_t^k (x^i,y,w_t)$ is the value of
\begin{align*}
\min_{u^{i}} \quad
& \coutint^{i}_{t}(x^{i},u^{i},{w}_{t})
  + \mu_{t}^{(k)}(y) \cdot \production^{i}_{t}(x^{i},u^{i},{w}_{t})
  + V_{t+1}^{i}\bp{{x}_{t+1}^{i},{y}_{t+1}^{i}} \\
\text{s.t.} \quad
& {x}_{t+1}^{i} = \dynamics^{i}_{t}(x^{i},u^{i},{w}_{t}) \eqfinv \\
& {y}_{t+1} = h_{t}(y,{w}_{t}) \eqfinp
\end{align*}

Storing the $\mathrm{argmin}$ obtained during the Bellman resolution,
we obtain the optimal feedback laws $\gamma_{t}^{i,(k)}$ as functions
of both the state $(x^{i},y)$ and the noise $w$ at time~$t$.
These functions allow to compute the optimal state and control
processes $\bp{\va{u}^{i,(k)},\va{x}^{i,(k)}}$ of subproblem~$i$
at iteration~$k$.\footnote{Remember that the process~$\va{y}$ follows
the non-controlled Markovian dynamics~\eqref{eq:dyn-info} and thus can be
obtained once for all.} Starting from $\va{x}_{0}^{i,(k)} = \va{x}_{0}^{i} $
% \begin{align*}
% \va{x}_{0}^{i,(k)} =
%  & \; \va{x}_{0}^{i} \eqfinv \\
% \intertext{
the optimal control and state variables are obtained by applying
the optimal feedback laws from~$t=0$ up to~$T-1$:
%}
\begin{align*}
\va{u}_{t}^{i,(k)} =
 & \; \gamma_{t}^{i,(k)}(\va{x}_{t}^{i,(k)},\va{y}_{t},\va{w}_{t}) \eqfinv \\
\va{x}_{t+1}^{i,(k)} =
 & \; f_{t}^{i}(\va{x}_{t}^{i,(k)},\va{u}_{t}^{i,(k)},\va{w}_{t}) \eqfinp
\end{align*}
\item The second step of the DADP algorithm consists in updating
the multiplier process $\va{\Lambda}^{(k)}$. Instead of updating
the multipliers themselves by the standard gradient formula\footnote{More
sophisticated formulas can be used in practice: see~\S\ref{sec:num-academic-convergence}.}

\begin{equation*}
\va{\Lambda}_{t}^{(k+1)} = \va{\Lambda}_{t}^{(k)} + \rho_t
\bgp{\sum_{i=1}^{N}\production_{t}^{i}\bp{\va{x}_t^{i,(k)},\va{u}_t^{i,(k)},\va{w}_{t}}}
\eqfinv
\end{equation*}
%\vl{$\rho$ instead of $\rho_t$ ?}

it is sufficient to deal with their conditional expectations w.r.t. $\va{y}_{t}$.
Using the optimal processes~$\va{x}^{i,(k)}$ and~$\va{u}^{i,(k)}$ obtained
at the previous step of the algorithm for all subproblems, the \emph{conditional}
deviation from the coupling constraint is obtained:
\begin{equation*}
\bgespc{\sum_{i=1}^{N}
        \production_{t}^{i}\bp{\va{x}_t^{i,(k)},\va{u}_t^{i,(k)},\va{w}_{t}}}
       {\va{y}_{t}} \eqfinp
\end{equation*}
By the Doob property, this conditional expectation can be represented
by a measurable mapping~$\Delta_{t}^{(k)}$:
\begin{equation}
\Delta_{t}^{(k)}({y}_{t}) =
\bgespc{\sum_{i=1}^{N}
        \production_{t}^{i}\bp{\va{x}_t^{i,(k)},\va{u}_t^{i,(k)},\va{w}_{t}}}
       {\va{y}_{t} = y_t} \eqfinp
\label{eq:Doob-bis}
\end{equation}
Gathering the functional representations~\eqref{eq:Doob} and~\eqref{eq:Doob-bis}
of the conditional multiplier and of the conditional deviation, the gradient
update reduces to the following functional expression:
\begin{equation}
\label{eq:fonct-multiplier-update}
\mu_{t}^{(k+1)}(\cdot) = \mu_{t}^{(k)}(\cdot) + \rho_t \Delta_{t}^{(k)}(\cdot) \eqfinp
\end{equation}
This last equation is equivalent to the multipliers conditional expectation update:
\begin{multline}
\label{eq:cond-multiplier-update}
\bespc{\va{\Lambda}_{t}^{(k+1)}}{\va{y}_{t}} =
\bespc{\va{\Lambda}_{t}^{(k)}}{\va{y}_{t}} \\
+ \rho_t \;
\Bespc{\sum_{i=1}^{N}\production_{t}^{i}\bp{\va{x}_t^{i,(k)},\va{u}_t^{i,(k)},\va{w}_{t}}}
      {\va{y}_{t}} \eqfinp
\end{multline}
\end{itemize}
DADP algorithm is depicted in Figure~\ref{fig:DADP}.

\tikzstyle{decision} = [diamond, draw, fill=blue!20,
    text width=7em, text badly centered, node distance=3cm, inner sep=0pt]
\tikzstyle{block} = [rectangle, draw, fill=blue!20,
    text width=6.8em, text centered, rounded corners, minimum height=4em]
\tikzstyle{line} = [draw, -latex']%TODO arrowhead
\tikzstyle{cloud} = [draw, ellipse,fill=red!20, node distance=3cm,
    minimum height=2em]

\begin{figure}[ht]
\scalebox{0.75}{
\begin{tikzpicture}[node distance = 3cm, auto]
    % Place nodes
    \node [block] (init) {Multiplier {function $\mu^{(k)}_t(y)$}};
    \node [block, below of=init, node distance = 3.7cm] (dots) {$\cdots$};
    \node [block, left of=dots](sub-1){Solving subproblem $1$:\\
    DP on {($\va{x}_t^1,\va{y}_t$)}};
    \node [block, right of=dots](sub-N){Solving subproblem $N$:\\
    DP on {($\va{x}_t^N,\va{y}_t$)}};
    \node [block, below of= dots, node distance = 3.7cm, text width=5.5cm](test)
    {{$\displaystyle \underbrace{\bgespc{\sum_{i=1}^N \Theta^i_t\bp{\cdot}}
    {\va y_t = y}}_{\phantom{\va\Delta_t^{(k)}}\Delta_t^{(k)}
    {(y)}\phantom{\va\Delta_t^{(k)}}}=0$} ?};
    \node [cloud, left of= sub-1, text width=4.5cm, node distance = 5cm](update)
    {\hbox{{$\displaystyle \!\!\!\!\!\!{\mu_{t}^{(k+1)}(\cdot)} =
    {\mu_{t}^{(k)}(\cdot)}+\rho_t \Delta_t^{(k)}{(\cdot)}$}}};
    % Draw edges
    \path [line] (init) -- (dots);
    \path [line] (init) -- (sub-1);
    \path [line] (init) -- (sub-N);
    \path [line] (sub-1) -- node [left] {{$\Theta_t^i\bp{\va x_t^{i,(k)},\va u_t^{i,(k)},\va{w}_t}$}} (test);
    \path [line] (dots) -- (test) ;
    \path [line] (sub-N) -- (test) ;
    \path [line] (test) -| (update) ;
    \path [line] (update) |- (init);
     \node [block, above of=init, text width=4.5cm, fill=green!20, node distance = 2cm](Y)
     {Information Process {$\va y_{t+1}=h_t(\va y_t,\va w_{t})$}};
     \path [line]  (Y) -- (init);
\end{tikzpicture}
}
\caption{\label{fig:DADP} DADP flowchart.}
\end{figure}
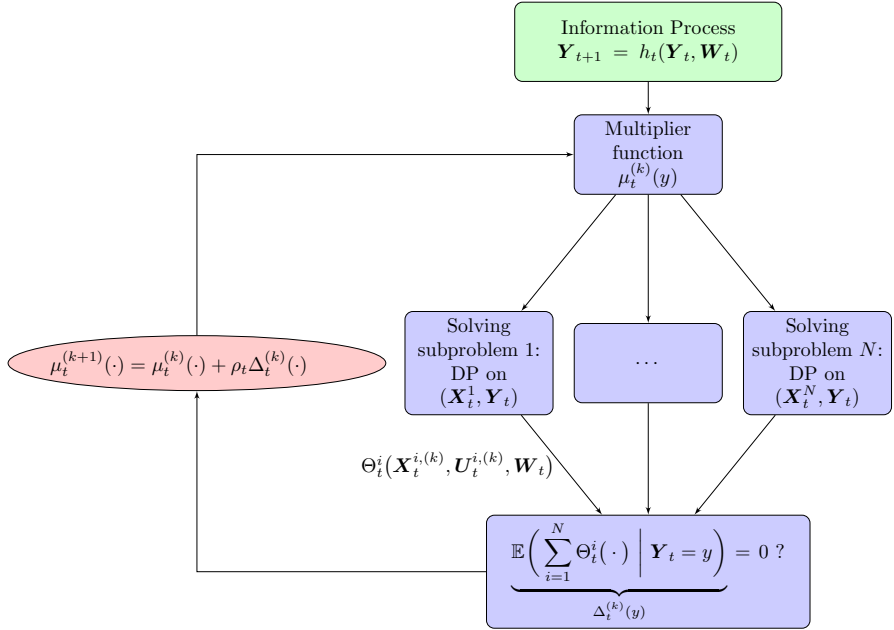

\subsection{DADP interpretations\label{sec:DADP-interpretation}}

The DADP method, as it has been presented up to now, consists in
an \emph{approximation} of the optimal multiplier, that is, the multiplier $\va{\Lambda}_{t}$
is replaced by its conditional expectation $\bespc{\va{\Lambda}_{t}}{\va{y}_{t}}$.
% \begin{equation*}
% \va{\Lambda}_{t} \quad \rightsquigarrow \quad \bespc{\va{\Lambda}_{t}}{\va{y}_{t}}
% \eqfinp
% \end{equation*}
Such an approximation  is equivalent to a \emph{decision-rule} approach for
the dual problem, obtained by imposing measurability conditions to the dual
variables~$\va{\Lambda}_{t}$:
% \footnote{The notation~$\va{x}\preceq\va{y}$ expresses
% that the random variable~$\va{x}$ is measurable w.r.t. the random variable~$\va{y}$.}
\begin{equation*}
\max_{\va{\Lambda}} \; \min_{\va{x},\va{u}} \;
  \mathcal{L}\bp{\va{x}, \va{u},\va{\Lambda}}
\quad \rightsquigarrow \quad
\max_{\va{\Lambda}_{t}\preceq \va{y}_{t}} \; \min_{\va{x},\va{u}} \;
  \mathcal{L}\bp{\va{x},\va{u},\va{\Lambda}} \eqfinp
\end{equation*}

DADP may also be viewed as a \emph{relaxation} of the constraints
in the primal problem. More precisely, we replace the almost sure coupling constraint~\eqref{pb:compact-cou} by a conditional expectation constraint, that is we
consider the following
relaxed version of the initial problem~\eqref{pb:compact}:
\begin{subequations}
\begin{equation}
\min_{(\va{x},\va{u})_{i\in\dc{1,N}}} \;
\bgesp{\sum_{i=1}^{N}
\Bp{\sum_{t=0}^{T-1}\coutint^{i}_{t}(\va{x}^{i}_{t},\va{u}^{i}_{t},\va{w}_{t})+
    \coutfin^{i}(\va{x}^{i}_{T})}} \eqfinv
\end{equation}
subject to the dynamics constraints~\eqref{pb:compact-dyn}, to the measurability
constraints~\eqref{pb:compact-mes} and to the \emph{conditional} coupling constraints:
\begin{equation}
\Bespc{\sum_{i=1}^{N}\production^{i}_{t}(\va{x}^{i}_{t},\va{u}^{i}_{t},\va{w}_{t})}
      {\va{y}_{t}} = 0 \eqfinp
\label{pb:SOCcond-cpl}
\end{equation}
\label{pb:SOCcond}
\end{subequations}

\begin{proposition}
Suppose the Lagrangian associated with Problem~\eqref{pb:SOCcond}
has a saddle point. Then the DADP algorithm can be interpreted as
the Uzawa algorithm applied to Problem~\eqref{pb:SOCcond}.
\end{proposition}
\begin{proof}
Consider the duality term
$\besp{\nespc{\va{\Lambda}_{t}^{(k)}}{\va{y}_{t}}\cdot
              \production^{i}_{t}(\va{x}^{i}_{t},\va{u}^{i}_{t},\va{w}_{t})}$
which appears in the cost function of subproblem~$i$ in DADP. This term
can be written equivalently $\besp{\va{\Lambda}_{t}^{(k)}\cdot
      \nespc{\production^{i}_{t}(\va{x}^{i}_{t},\va{u}^{i}_{t},\va{w}_{t})}
             {\va{y}_{t}}}$.
% SHORTER
% \begin{equation*}
% %\besp{\nespc{\va{\Lambda}_{t}^{(k)}}{\va{y}_{t}}\cdot
% %      \production^{i}_{t}(\va{x}^{i}_{t},\va{u}^{i}_{t},\va{w}_{t})} =
% \besp{\va{\Lambda}_{t}^{(k)}\cdot
%       \nespc{\production^{i}_{t}(\va{x}^{i}_{t},\va{u}^{i}_{t},\va{w}_{t})}
%              {\va{y}_{t}}} \eqfinp
% \end{equation*}
which corresponds to the dualization of
%the constraint:
%\begin{equation*}
% \Bespc{\sum_{i=1}^{N}\production^{i}_{t}(\va{x}^{i}_{t},\va{u}^{i}_{t},\va{w}_{t})}
%       {\va{y}_{t}} = 0 \eqfinv
%\end{equation*}
%that is,
the coupling constraint handled by Problem~\eqref{pb:SOCcond}.
\end{proof}

DADP thus consists in replacing an \emph{almost-sure}
constraint by its \emph{conditional expectation} w.r.t.
the information variable~$\va{y}_{t}$. From this interpretation,
we deduce that the optimal value provided by DADP is a guaranteed
\emph{lower bound} of the optimal value of Problem~\eqref{pb:compact}.

\subsection{Admissibility recovery\label{sec:DADP-admissibility}}

Instead of solving the original problem~\eqref{pb:compact}, DADP deals
with the relaxed problem~\eqref{pb:SOCcond} in which the almost-sure
coupling constraint~\eqref{pb:compact-cpl} is replaced by the less binding
constraint~\eqref{pb:SOCcond-cpl}. As a consequence, a solution
of Problem~\eqref{pb:SOCcond} does not satisfy the set of constraints
of Problem~\eqref{pb:compact}. An additional procedure has to be devised
in order to produce an (at last) admissible solution of~\eqref{pb:compact}.

Nevertheless, solving Problem~\eqref{pb:SOCcond} produces at each time~$t$
a set of~$N$ local Bellman functions~$V_{t}^{i}$, each depending on the
extended state~$(x_{t}^{i},y_{t})$. We use these functions to produce
a single Bellman function~$\widehat{V}_{t}$ depending on the global
state~$\bp{x_{t}^{1},\ldots,x_{t}^{N}}$ which is used as an
\emph{approximation} of the ``true'' Bellman function~$V_{t}$ of
Problem~\eqref{pb:compact}. The heuristic rule leading to this
approximated Bellman function simply consists in summing the local
Bellman functions:
\begin{equation*}
\widehat{V}_{t}\bp{x_{t}^{1},\ldots,x_{t}^{N},y_{t}} =
\sum_{i=1}^{N} V_{t}^{i} \bp{x_{t}^{i},y_{t}} \eqfinp
\end{equation*}
The approximated Bellman functions~$\widehat{V}_{t}$ allow us
to devise an admissible feedback policy for Problem~\eqref{pb:compact}:
for any value of the state~$\bp{x_{t}^{1},\ldots,x_{t}^{N}}$, any value
of the information~$y_{t}$ and any value of the noise~$w_{t}$ at time~$t$,
the control value is obtained by solving the following one-step DP problem
\begin{align*}
\min_{(u_{t}^{1},\ldots,u_{t}^{N})} \quad
& \; \sum_{i=1}^{N} L_{t}^{i}\bp{x_{t}^{i},u_{t}^{i},w_{t}^{i}} +
  \widehat{V}_{t+1}\bp{x_{t+1}^{1},\ldots,x_{t+1}^{N},y_{t+1}} \eqfinv \\
% \intertext{subject to the constraints}
\text{s.t.} \quad
& x_{t+1}^{i} = f_{t}^{i} \bp{x_{t}^{i},u_{t}^{i},w_{t}^{i}} \eqsepv i\in\dc{1,N} \eqfinv \\
& y_{t+1} = h_{t} \bp{y_{t}^{i},w_{t}^{i}} \eqfinv \\
 & \sum_{i=1}^{N} \production^{i}_{t}(x^{i}_{t},u^{i}_{t},w_{t}) = 0 \eqfinp
\end{align*}
In this framework, DADP can be viewed as a tool allowing to compute approximated
Bellman functions for Problem~\eqref{pb:compact}. These functions are then used
to compute control values satisfying all the constraints of the problem, that is,
to produce an online admissible feedback policy for Problem~\eqref{pb:compact}.

Applying this online feedback policy along a bunch of noises scenarios allows
to compute a Monte Carlo approximation of the cost, which is accordingly
a stochastic \emph{upper bound} of the optimal value of Problem~\eqref{pb:compact}.

\subsection{Theoretical and practical questions\label{sec:DADP-questions}}

The theoretical questions linked to DADP are adressed in \cite{TheseLeclere},
and the practical ones in \cite{TheseGirardeau}and \cite{TheseAlais}.

\subsubsection{Theoretical questions\label{sec:DADP-questions-theoretical}}

In the DADP approach, we treat the coupling constraints of a stochastic
optimization problem by duality methods and solve it using the Uzawa
algorithm. Uzawa algorithm is a dual method which is usually described
in an Hilbert space such as $\mathrm{L}^{2}(\omeg,\trib,\prbt,\mathbb{R}^{n})$,
but we cannot guarantee the existence of an optimal multiplier in such a space.
To overcome the difficulty, the approach consists in extending the setting to
the non-reflexive Banach space $\mathrm{L}^{\infty}(\omeg,\trib,\prbt,\mathbb{R}^{n})$,
to give conditions for the existence of an optimal multiplier in
$\mathrm{L}^{1}\bp{\omeg,\trib,\prbt;\mathbb{R}^n}$ (rather than in the dual space
of $\mathrm{L}^{\infty}$) and to study the Uzawa algorithm convergence in this space.
The interested reader is referred to \cite{TheseLeclere} for more information.

\subsubsection{Practical questions\label{sec:DADP-questions-practical}}

An important practical question is the choice of the information variables $\va{y}_{t}$.
We present here some possibilities.
\begin{enumerate}
\item \emph{Perfect memory}:
$\va{y}_{t}=\bp{\va{w}_{0},\ldots,\va{w}_{t}}$. \\
From the measurability properties of~$\va{\Lambda}_{t}^{(k)}$, we have
$\nespc{\va{\Lambda}_{t}^{(k)}}{\va{y}_{t}} = \va{\Lambda}_{t}^{(k)}$,
that is, there is no approximation! A valid state for each subproblem is
then $\bp{\va{w}_{0},\ldots,\va{w}_{t}}$: the state is growing with time.
\item \emph{Minimal information}:
$\va{y}_{t} =0$.\footnote{or equivalently $\va{y}_{t}$ being
any \emph{constant} random variable} \\
Here $\va{\Lambda}_{t}^{(k)}$ is approximated by its expectation
$\nesp{\va{\Lambda}_{t}^{(k)}}$. The information variable does not deliver
any online information, and a valid state for subproblem ~$i$ is $\va{x}_{t}^{i}$.
\item \emph{Dynamic information}:
$\va{y}_{t+1}=h_{t}\bp{\va{y}_{t},\va{w}_{t+1}}$. \\
This choice corresponds to a number of possibilities, as mimicking the state
of another unit, or adding a hidden dynamics. A valid state for subproblem~$i$
is $\bp{\va{x}_{t}^{i},\va{y}_{t}}$.
\end{enumerate}

Finally, the question of accelerating the DADP algorithm by replacing
the standard Lagrangian by an augmented one, or by using more sophisticated
methods than the simple gradient ascent method in the multiplier
update step, have a great interest in order to improve the method.
This point is developed in~\S\ref{sec:num-academic-convergence}.

%!TEX root = Article_Barrages.tex
%%%%%%%%%%%%%%%%%%%%%%%%%%%%%%%%%%%%%%%%%%%%%%%%%%%%%%%%%%%%%%%%%%%%%%
%                                                                    %
% 4. Numerical experiments                                           %
%                                                                    %
%%%%%%%%%%%%%%%%%%%%%%%%%%%%%%%%%%%%%%%%%%%%%%%%%%%%%%%%%%%%%%%%%%%%%%

\section{Numerical experiments\label{sec:numerical}}

In this section, we present numerical results obtained on a large selection
of hydro valleys. Some of these valleys (see Figure~\ref{fig:academic})
correspond to academic examples, in the sense that their characteristics
(size of dams, range of controls, inflows values) do not rely on existing
valleys. These examples allow us to quantify the performance of different
optimization methods (DP, DADP and two different flavors of SDDP) on problems
of increasing size, from a valley incorporating 4 dams, and thus solvable
by DP, up to a valley with 30 dams, and thus facing the curse of dimensionality
(\S\ref{sec:num-academic} and \S\ref{sec:num-challenge}). We also present
two instances corresponding to more realistic hydro valleys, where
the models respect the orders of magnitude of the dam sizes of existing valleys
(\S\ref{sec:num-realistic}).

All the results presented here have been obtained using a 3.4GHz, 4~cores -- 8~threads
Intel$\circledR$ Xeon$\circledR$ E3 based computer.%\fp{Donner la fr\'{e}quence du processeur}\pc{3.4GHz}

\subsection{Application of DADP to a hydro valley\label{sec:num-hydro}}

We go back to the problem formulation~\eqref{pb:normal} presented
at \S\ref{sec:frame-problem}. In order to implement the DADP algorithm,
we dualize the coupling constraints
\begin{equation}
\label{eq:coupling-i-t}
\zipt - g_{t}^{i}(\xit,\uit,\wit,\zit) = 0 \; ,
\end{equation}
and we denote by $\va{\Lambda}_{t}^{i+1}$ the associated multiplier
(random variable).

When minimizing the dual problem at iteration~$k$ of the algorithm,
the product with a given multiplier by $\va{\Lambda}_{t}^{i+1,(k)}$
\begin{equation*}
\va{\Lambda}_{t}^{i+1,(k)} \cdot \bp{\zipt-g_{t}^{i}(\xit,\uit,\wit,\zit)} \; ,
\end{equation*}
is additive with respect to the dams, that is,
\begin{itemize}
\item the first term $\va{\Lambda}_{t}^{i+1,(k)} \cdot \zipt$
      pertains to dam~$i\!+\!1$ subproblem,
\item whereas the second term
      $\va{\Lambda}_{t}^{i+1,(k)} \cdot g_{t}^{i}\bp{\xit,\uit,\wit,\zit}$
      pertains to dam~$i$ subproblem,
\end{itemize}
hence leading to a dam by dam decomposition for the dual problem maximization
in~$(\va{x},\va{u},\va{z})$ at~$\va{\Lambda}_{t}^{i+1,(k)}$ fixed.
\begin{figure}[ht]
\begin{center}
  \includegraphics[width=0.5 \textwidth]{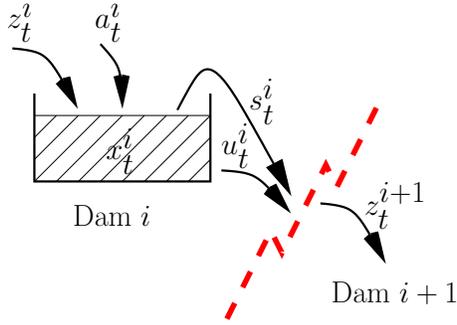}
\end{center}
\caption{\label{fig:Decomposition} Decomposition by dam.}
\end{figure}

\subsubsection{DADP implementation\label{sec:num-hydro-dadp}}

The DADP method consists in choosing a multiplier process $\va{y}$
and then replacing the coupling constraints by their conditional
expectations with respect to~$\va{y}_{t}$. Here we adopt the choice
%\begin{equation}
%\label{eq:valley-information}
% SHORTER
$\va{y}_{t} = 0$
%\; ,
%\end{equation}
(minimal information), so that Constraints~\eqref{eq:coupling-i-t}
are replaced in the approximated problem by their expectations:
\begin{equation}
\label{eq:valley-constraint}
\besp{\zipt-g_{t}^{i}(\xit,\uit,\wit,\zit)} = 0 \eqfinp
\end{equation}
The expression of Subproblem~\eqref{pb:subproblem-approx}
attached to dam~$i$ reads
\begin{subequations}
\label{eq:valley-subproblem}
% \begin{multline}
% \max_{\va{u}^{i},\va{z}^{i},\va{x}^{i}}
%   \bgesp{ \sum_{t=0}^{T-1} \Bp{\coutint_{t}^{i}\bp{\xit,\uit,\wit,\zit}
%   \; + \; \besp{\va{\Lambda}_{t}^{i,(k)}} \cdot \zit \\
%   - \; \besp{\va{\Lambda}_{t}^{i+1,(k)}}
%   \cdot g_{t}^{i}\bp{\xit,\uit,\wit,\zit}} + \coutfin^{i}\bp{\xitf}} \eqfinv
% \end{multline}

\begin{align}
\max_{\va{u}^{i},\va{z}^{i},\va{x}^{i}} \quad &
  \bgesp{ \sum_{t=0}^{T-1} \Bp{\coutint_{t}^{i}\bp{\xit,\uit,\wit,\zit}
  \; + \; \besp{\va{\Lambda}_{t}^{i,(k)}} \cdot \zit \\
  & \qquad - \; \besp{\va{\Lambda}_{t}^{i+1,(k)}}
  \cdot g_{t}^{i}\bp{\xit,\uit,\wit,\zit}} + \coutfin^{i}\bp{\xitf}} \eqfinv
\qquad\qquad \\
s.t. \qquad &  \va{x}^{i}_{t+1} = \dynamics^{i}_{t}(\va{x}^{i}_{t},\va{u}^{i}_{t},\va{w}_{t})
  \eqfinv  \qquad \va{x}^{i}_{0} \quad \text{given}\\
& \va{u}^{i}_{t} \;\;\: \preceq \sigma(\va{w}_{0},\ldots,\va{w}_{t}) \eqfinp
\end{align}
\end{subequations}
Because of the crude relaxation due to a constant~$\yit$, the multipliers $\va{\Lambda}_{t}^{i,(k)}$ appear only in the subproblems by means
of their expectations $\besp{\va{\Lambda}_{t}^{i,(k)}}$, so that all
subproblems involve a $1$-dimensional state variable, that is, the dam
stock~$\va{x}^{i}_{t}$, and hence are easily solvable by Dynamic Programming.

We denote by $\bp{\va{u}^{i,(k)},\va{z}^{i,(k)},\va{x}^{i,(k)}}$
the optimal solution of each subproblem~$i$, and by
%\begin{equation*}
$V_{t}^{i,(k)}(x^{i})$% \eqfinv
%\end{equation*}
 the Bellman function obtained for each dam~$i$ at time~$t$.

With the choice of constant information variables~$\yit$,
the coordination update step~\eqref{eq:cond-multiplier-update} reduces to
\begin{equation}
\label{eq:valley-update}
\besp{\va{\Lambda}_{t}^{i,(k+1)}} = \besp{\va{\Lambda}_{t}^{i,(k)}}
+ \rho_t \Besp{\va{z}^{i+1,(k)}_{t} - g_{t}^{i}\bp{\va{x}^{i,(k)}_{t},\va{u}^{i,(k)}_{t},\wit,\va{z}^{i,(k)}_{t}}}
\eqfinv
\end{equation}
that is, a collection of deterministic equations involving the expectation
of~\eqref{eq:coupling-i-t} which is easily computable by a Monte Carlo approach.

Assume that DADP converges, leading to optimal Bellman
functions~$V_{t}^{i,\infty}$ and to optimal solutions
$\bp{\va{u}^{i,\infty},\va{z}^{i,\infty},\va{x}^{i,\infty}}$.
We know that the initial almost-sure coupling constraints are not
satisfied. To recover admissibility, the heuristic rule
proposed at \S\ref{sec:DADP-admissibility} consists in forming
the global approximated Bellman $V_{t}^{\infty}$ functions as
\begin{equation*}
V_{t}^{\infty}\bp{x^{1},\ldots,x^{N}} =
\sum_{i=1}^{N} V_{t}^{i,\infty} \bp{x^{i}} \eqfinv
\end{equation*}
and then computing, at any time $t$ and for any pair $(x_t,w_t)$,
a control satisfying all the constraints of Problem~\eqref{pb:normal}
by solving the following one-step DP problem:
\begin{align*}
\max_{(u^{1},\ldots,u^{N})} \;\;
  & \sum_{i=1}^{N} L_{t}^{i}\bp{x^{i},u^{i},w_{t}^{i},z^{i}} +
    V_{t+1}^{\infty} \bp{x_{t+1}^{1},\ldots,x_{t+1}^{N}} \eqfinv \\
s.t. \quad\;\;
  & x_{t+1}^{i} = f_{t}^{i} \bp{x^{i},u^{i},w_{t}^{i},z^{i}} \quad \forall i \eqfinv \\
  & z^{i+1} = g_{t}^{i}(x^i,u^{i},w_{t}^{i},z^{i}) \quad \; \forall i \eqfinp
\end{align*}

\subsubsection{Complete process\label{sec:num-hydro-summary}}

To summarize, the whole process for implementing DADP is as follows.

\begin{itemize}
  \item \textbf{Optimization stage}
  \begin{itemize}
    \item Apply the DADP algorithm, and obtain at convergence the local Bellman
    functions~$V_{t}^{i,\infty}$.
    \item Form the approximated global Bellman functions~$V_{t}^{\infty}$.
  \end{itemize}
  \item \textbf{Simulation stage}
  \begin{itemize}
    \item Draw a large number of noise scenarios (Monte Carlo sampling).
    \item Compute the control values along each scenario by solving
    the one-step DP problems involving the Bellman
    functions~$V_{t}^{\infty}$'s, thus satisfying all the constraints
    of Problem~\eqref{pb:normal} as explained in~\S\ref{sec:DADP-admissibility};
    these computations produce for each scenario a
    state and control trajectories, as well as a payoff.
    \item Evaluate the quality of the solution: trajectories variability,
    payoff distribution and associated mean\ldots
  \end{itemize}
\end{itemize}

\subsection{SDDP implementations\label{sec:num-sddp}}

As will be explained in~\S\ref{sec:num-academic}, the controls
of the original problem are discrete, which is a priori a difficulty
for SDDP implemetation.

We describe here the two implementations of SDDP that we use for numerical
comparisons, that is, a classical version of SDDP in which the integrity constraints
on the control variables are relaxed (continuous controls) in order to use standard
quadratic programming, and a homemade discrete version of SDDP described
later on.

\subsubsection{Continuous version of SDDP\label{sec:sddpc}}

The continuous implementation \SDDPc\ of SDDP relaxes the integrity constraints
upon the control $\va u$, and assume that for all time $t$ and all dam~$i$
\begin{equation}
  x_t^i  \in [\underline{x}^i,\; \overline{x}^i] \quad\quad
  u_t^i  \in [\underline{u}^i,\; \overline{u}^i] \eqfinp
\end{equation}
Furthermore, \SDDPc\ considers that the spillage
is a control variable, so as to render the dynamic linear. As the dynamic
is linear and costs convex, we are in the standard framework of Stochastic
Dual Dynamic Programming and the resolution converges asymptotically
to the optimal solution of the relaxed problem \cite{Girardeau_MOR_2015}.

The whole process of \SDDPc\ is as follows.
\begin{itemize}
  \item \textbf{Optimization stage.}
    The implementation of \SDDPc\ corresponds to the one described
    in \cite{Shapiro10}, that is, the so-called DOASA implementation.
    Lower approximations of the Bellman functions $V_t$ are built
    iteratively. At iteration $k$, the procedure consists of two passes.
    \begin{itemize}
      \item During the \emph{forward pass}, we first sample a scenario of noise.
      We then simulate a stock trajectory by using the current approximation
      of the Bellman functions. This is done by successively solving one-step
      DP problem to determine the next stock value. Each of these one-step DP
      problems is in fact a continuous quadratic programming (QP) problem.
      \item In the \emph{backward pass}, duality theory allows to find subgradient of lower approximations of the Bellman functions. This subgradients are used to construct valid cut, that is hyperplanes that are lower that the Bellman functions. Those cuts are then added to the current approximations of Bellman functions.
      % the approximations of the Bellman
      % functions are refined along the trajectory computed during the forward
      % pass.\pc{Pour FP : en dire plus sur la mani\`{e}re de raffiner.}
    \end{itemize}
  \item \textbf{Simulation stage.}
    The simulation stage is identical to the one described
    at \S\ref{sec:num-hydro-summary}: the controls are
    computed with a one-step DP problem using the approximation
    of the Bellman functions obtained by \SDDPc.
\end{itemize}

The continuous version \SDDPc\ is implemented in Julia, with the package
StochDynamicProgramming\footnote{See the github link
\url{https://github.com/JuliaOpt/StochDynamicProgramming.jl}.} built
on top of the JuMP package
used as a modeller \cite{DunningHuchetteLubin2015}. The QP problems are solved
using CPLEX 12.5. Every 5 iterations, redundant cuts are removed thanks to
the limited memory level-1 heuristic
described in \cite{guigues2017dual}. Indeed, without cuts removal, the resolution
of each QP becomes too slow as the number of cuts increase along iterations.
The algorithm is stopped after a fixed number of iterations and the gap is
 estimated with Monte-Carlo as described in \cite{Shapiro10}.

\subsubsection{Discrete version of SDDP\label{sec:sddpd}}
%\vlil{Je pense que l'on est trop attaquable sur cet approche pour la laisser dans la version ejor.}

In the discrete version of \SDDPd, the controls $\va u$ are discrete variables
which have the same constraints as in the original problem formulation.
Some other works mixing SDDP and binary variables are described in~\cite{zou2016nested}
The whole process of \SDDPd\ is as follows.
\begin{itemize}
  \item \textbf{Optimization stage.}
    The implementation of \SDDPd\ corresponds to the one described
    for  \SDDPc\ that is, approximations of the Bellman functions $V_t$
    by a set of cuts which are built iteratively. At iteration $k$,
    the procedure consists of two passes.
    \begin{itemize}
      \item During the \emph{forward pass}, we first sample a scenario of noise.
      We then simulate a stock trajectory by using the current approximation
      of the Bellman functions. This is done by successively solving one-step DP
      problem to determine the next stock. Each of these one-step DP problems
      is solved by enumerating all possible values of the discrete control.
      In order to take advantage of the threading facilities, we may treat
      multiple noise scenarios in parallel (8 scenarios in practice).
      \item In the \emph{backward pass}, the approximations of the Bellman
      functions are refined along the trajectories computed during the forward
      pass. For each new added state point we first solve a one-step DP problem
      to obtain a new Bellman value. Then we solve a family of one step DP problem
      using as state points the neighbors on the grid of the current state point.
      We are thus able to approximate a new cut at the current state point by
      approximating derivatives of the Bellman function by finite differences
      (we use centered finite differences for regular states and forward,
      backward differences for states at domain boundary).
    \end{itemize}
  \item \textbf{Simulation stage.}
    The simulation stage is identical to the one described
    at \S\ref{sec:num-hydro-summary}: the controls are
    computed with a one-step DP problem using the approximation
    of the Bellman functions obtained by \SDDPd.
\end{itemize}
The algorithm is stopped after a fixed number of iterations (25 in our experiments).
The number of added points at forward step as already explained was set to 8.
During the iterations some cuts were dropped by only keeping the last 100 cuts
obtained for each Bellman function.

\subsection{Results obtained for academic valleys\label{sec:num-academic}}

We model a first collection of hydro valleys including from 4 to 12 dams,
with arborescent geometries (see Figure~\ref{fig:academic}).

\begin{figure}[ht!]
\begin{center}
  \includegraphics[scale=0.225]{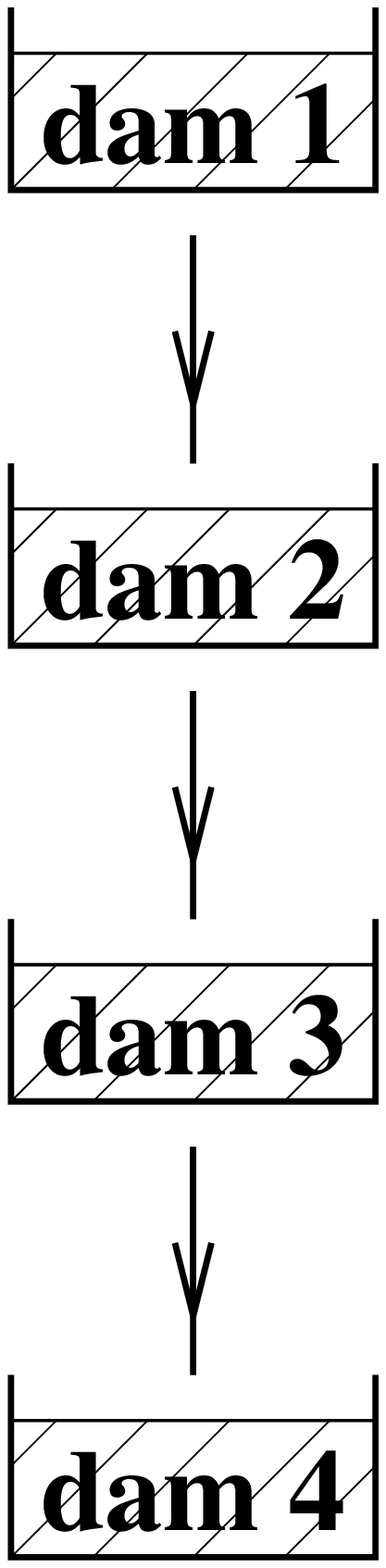} \hspace{0.1cm} $\;$
  \includegraphics[scale=0.225]{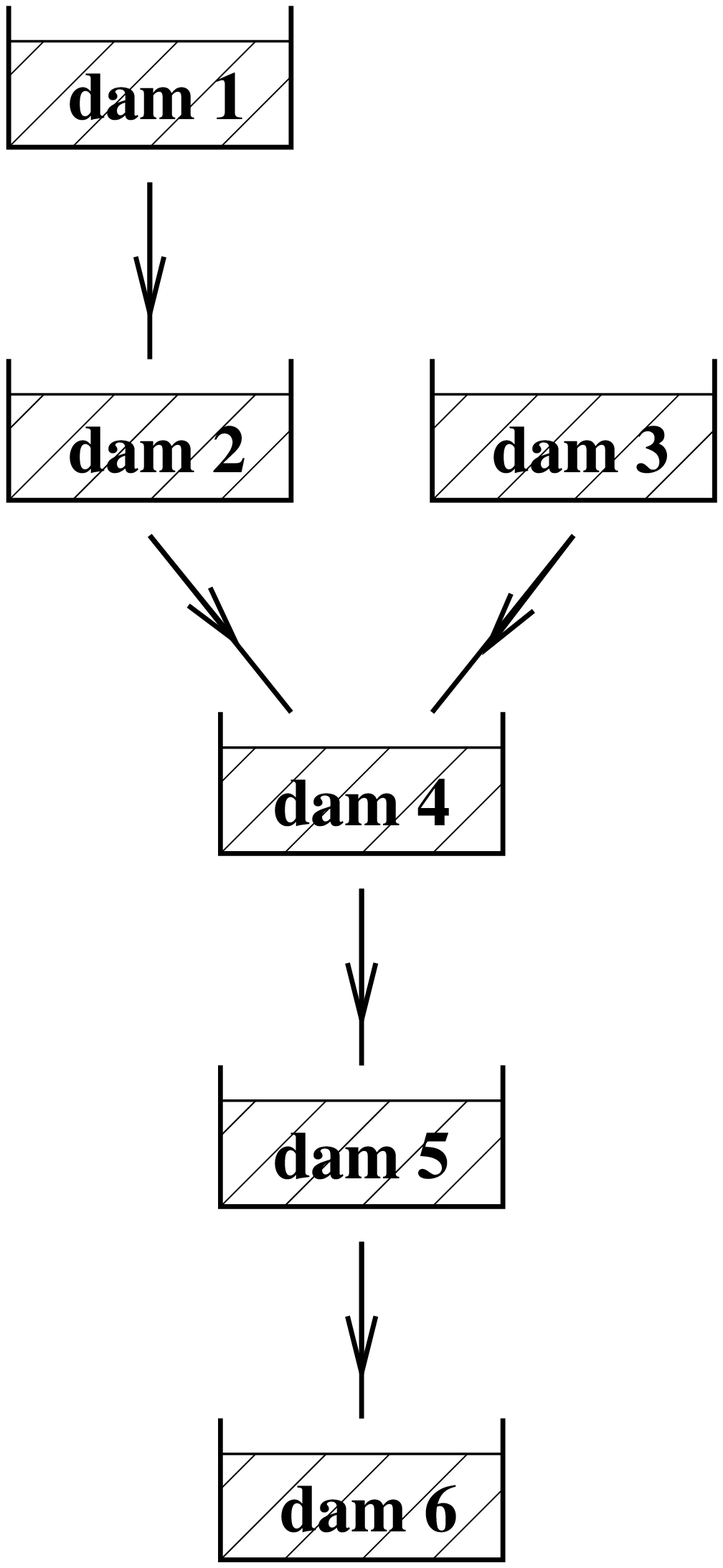} \hspace{0.1cm} $\;$
  \includegraphics[scale=0.225]{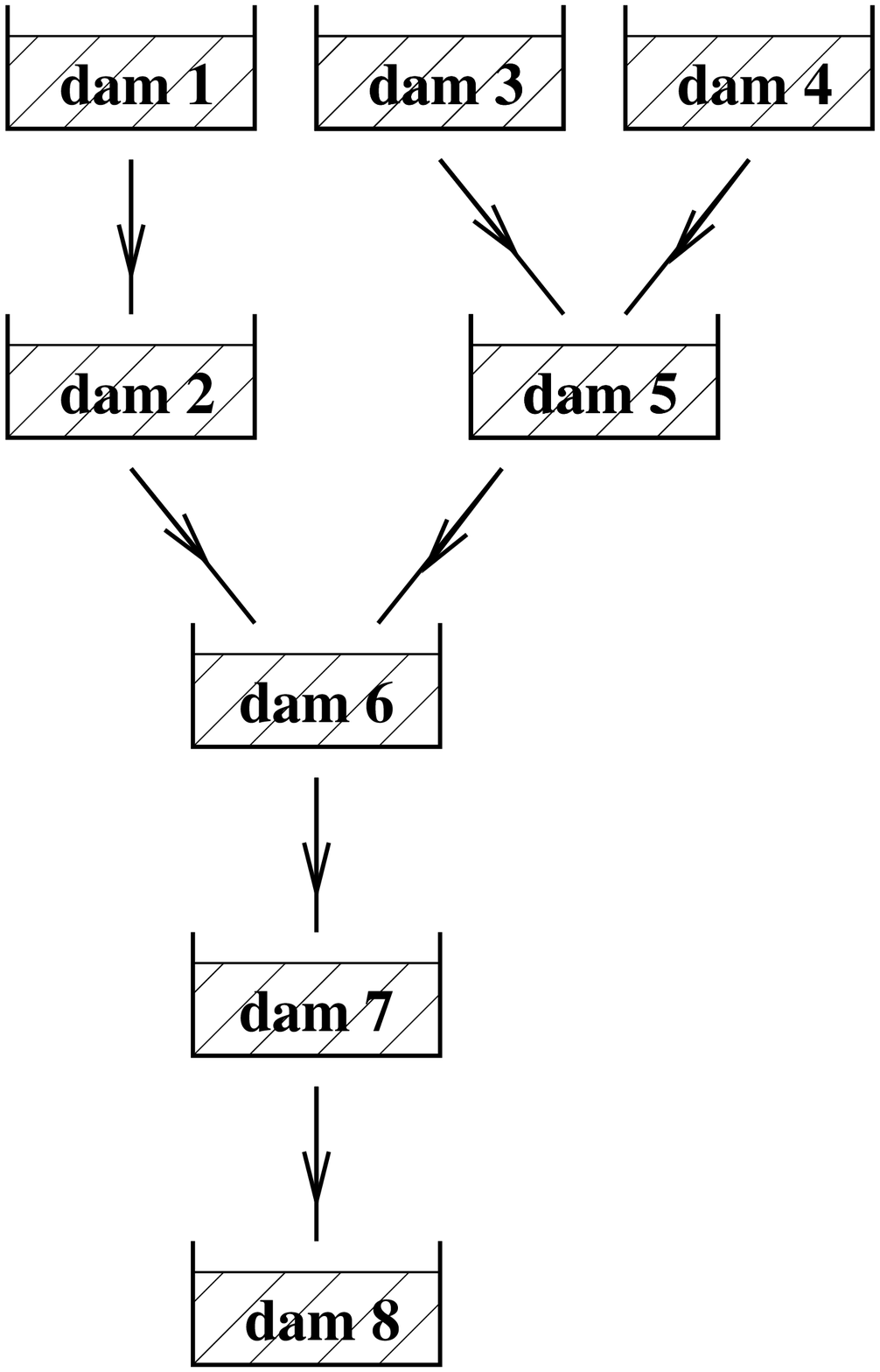} \hspace{0.1cm} $\;$
  \includegraphics[scale=0.225]{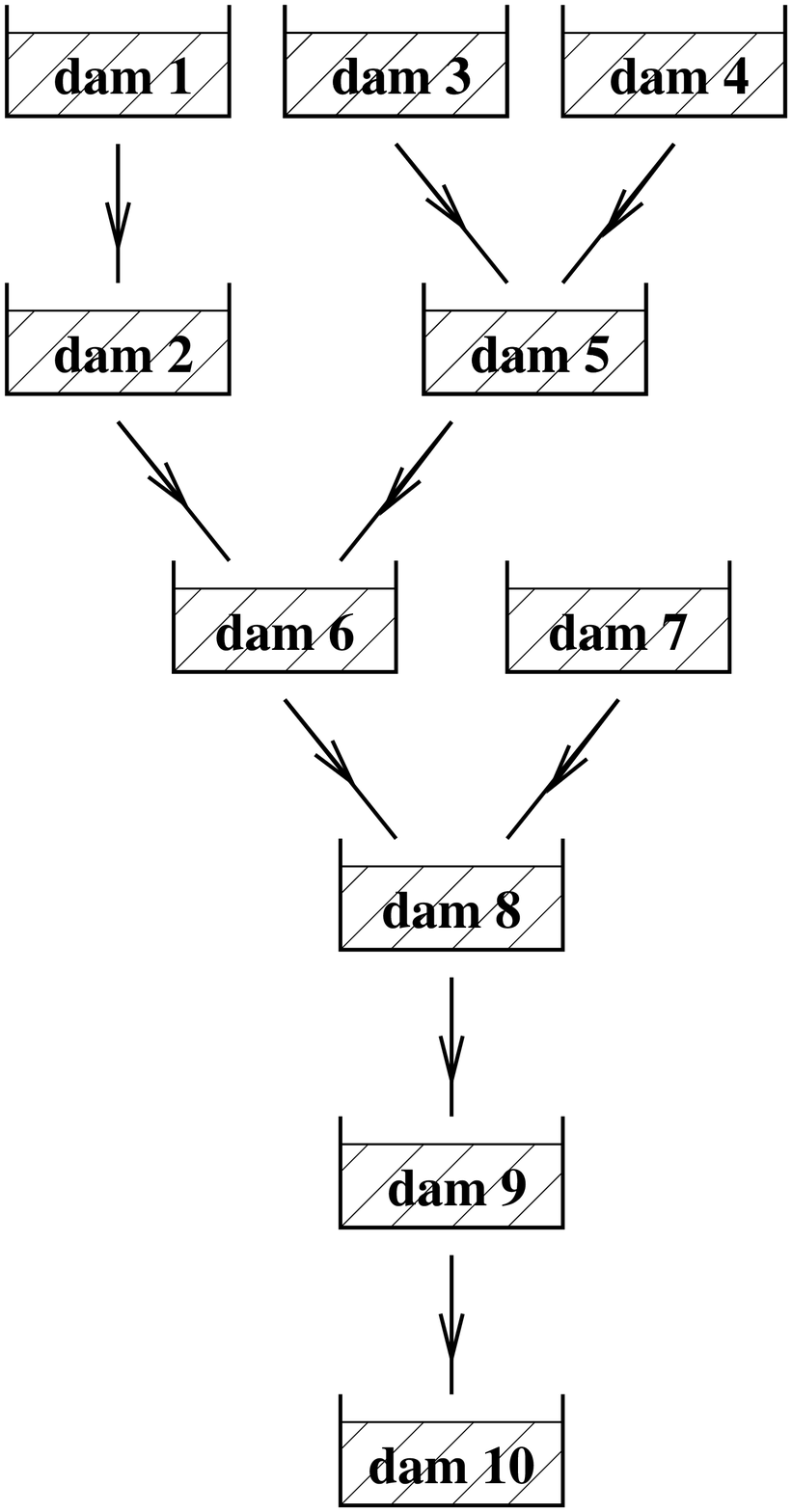}
\end{center}
\vspace{-0.2cm}
\hspace{0.2cm} 4-Dams \hspace{0.7cm} 6-Dams \hspace{1.4cm} 8-Dams \hspace{3.0cm} 10-Dams

\caption{\label{fig:academic} Some academic examples of hydro valleys.}
\end{figure}

The optimization problem is stated on a time horizon of one year, with
a monthly time step ($T=12$). All the dams have more or less the same maximal
volume. The maximal amount of turbinated water for each dam varies with
the location of the dam in the valley (more capacity for a downstream dam than
for an upstream dam), as well as the random inflows in a dam (more inflow for
an upstream dam than for a downstream dam). We assume discrete probability laws
with finite supports for the inflows,\footnote{Market prices are assumed to be
deterministic.} and we also assume that the available turbine controls are discrete,
so that each dam is in fact modeled using a discrete Markov chain.
%\fp{En dire plus sur la discr\'{e}tisation~?}
These valleys do not correspond to realistic valleys, in the sense that
a true valley incorporates dams with very heterogeneous sizes.

\subsubsection{DADP convergence\label{sec:num-academic-convergence}}

Let us first detail the gradient method used for the update of the multipliers
involved by \DADP. Thanks to the choice of constant information variables,
the gradient expression involved in the update formula~\eqref{eq:valley-update}
is an expectation:
\begin{equation*}
\Besp{\va{z}^{i+1,(k)}_{t} -
      g_{t}^{i}\bp{\va{x}^{i,(k)}_{t},\va{u}^{i,(k)}_{t},\wit,\va{z}^{i,(k)}_{t}}}
\eqfinp
\end{equation*}
This expectation can be approximated by a Monte Carlo approach.
We draw a collection of scenarios\footnote{Note that this
collection of scenarios has nothing to do with the one used
during the simulation stage of the complete process described
at \S\ref{sec:num-hydro-summary}.} of~$\na{\va{w}_{t}}$
and then compute at iteration $k$ of DADP the optimal solutions
$\ba{\va{x}^{i,(k)}_{t},\va{u}^{i,(k)}_{t},\va{z}^{i,(k)}_{t}}$
of Subproblem~\eqref{eq:valley-subproblem} along each scenario.
We thus obtain realizations of $\bp{\va{z}^{i+1,(k)}_{t} -
g_{t}^{i}\np{\va{x}^{i,(k)}_{t},\va{u}^{i,(k)}_{t},\wit,\va{z}^{i,(k)}_{t}}}$,
whose arithmetic mean gives the (approximated) gradient component
at time $t$ for the coupling at dam $i+1$ . This gradient can
be used either in the standard steepest descent method such as
in~\eqref{eq:valley-update}, or in a more sophisticated algorithm such as
the conjugate gradient or the quasi-Newton method. We use in our numerical
experiments a solver (limited memory BFGS) of the MODULOPT
library from INRIA \cite{gilbert2007libopt}.
For all the valleys we studied, the convergence was fast (around 100 iterations
regardless of the problem size). Figure~\ref{fig:dam4-multipliers} represents
the evolution of the multipliers of dam connections for the 4-Dams valley
along the iterations of the algorithm.

\begin{figure}[ht!]
\begin{center}
\includegraphics[scale=0.275]{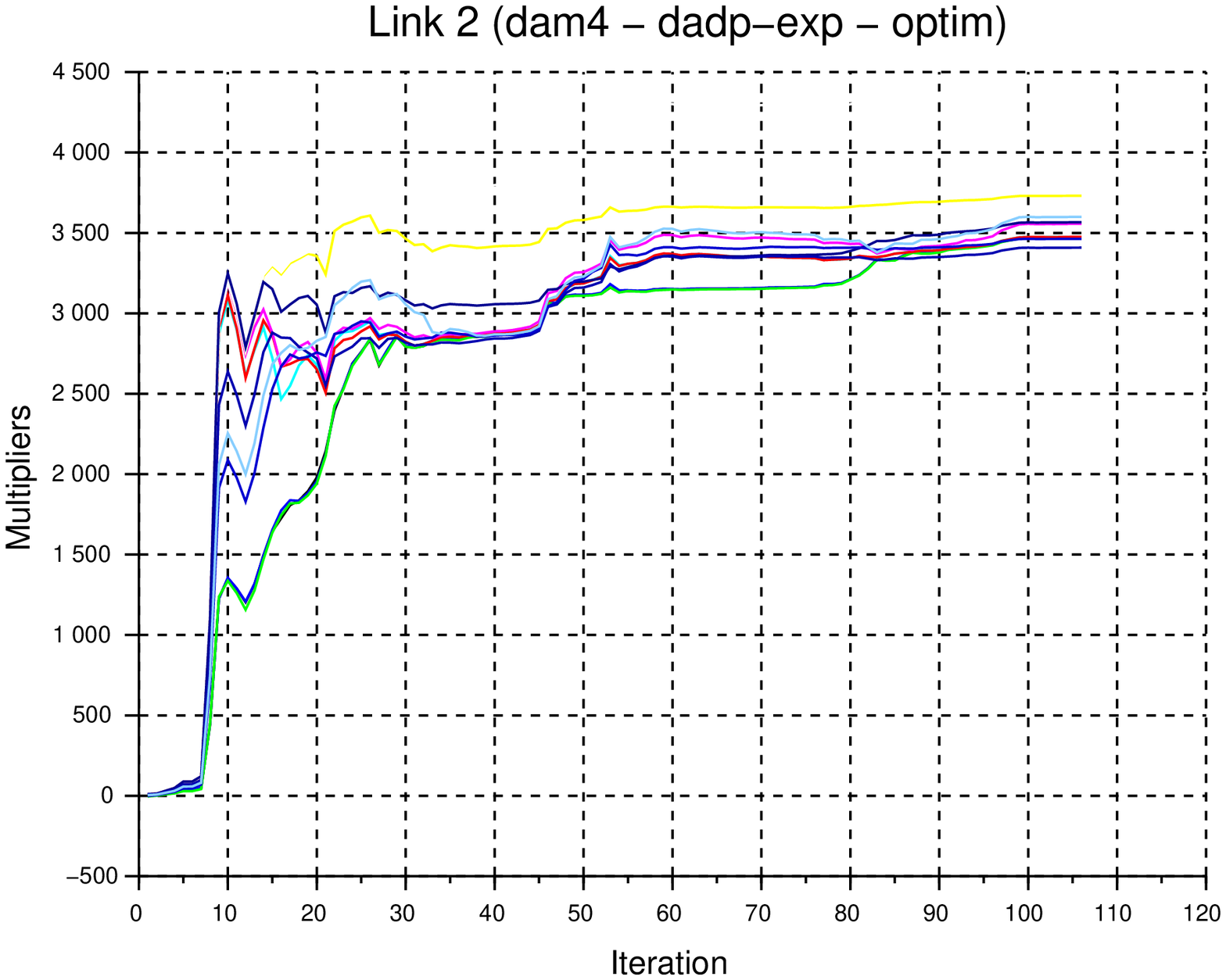}
$\!\!\!\!\!\!\!\!\!\!\!\!\!\!\!$
\includegraphics[scale=0.275]{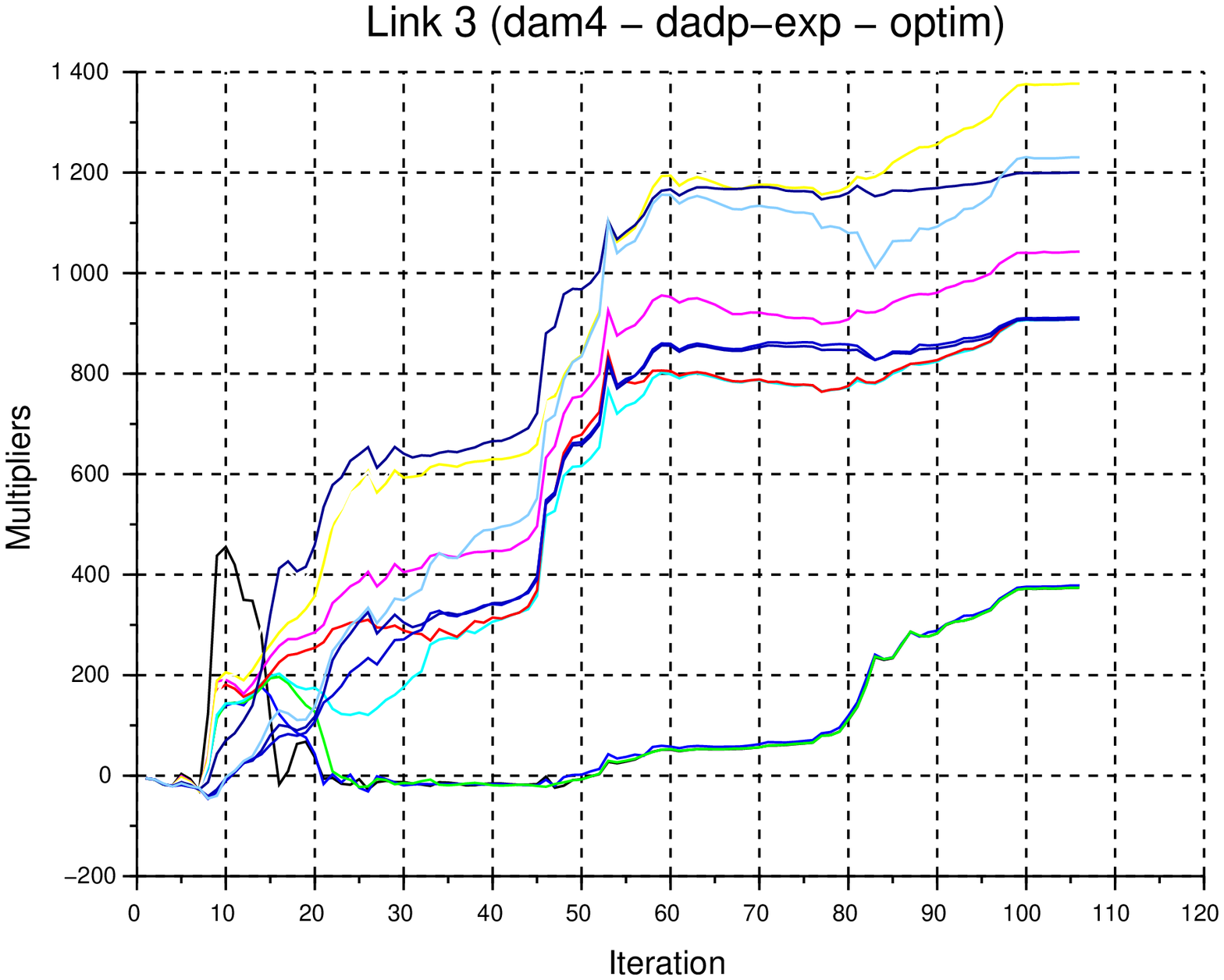}
\end{center}
\caption{4-Dams multipliers: \texttt{dam}1$\leftrightarrow$\texttt{dam}2 (left) ~---~
                                    \texttt{dam}2$\leftrightarrow$\texttt{dam}3 (right)}
\label{fig:dam4-multipliers}
\end{figure}

\subsubsection{Methods comparison\label{sec:num-academic-comparison}}

We solve Problem~\eqref{pb:normal} for the first collection of academic
valleys by 4 different methods:
\begin{enumerate}
\item the standard Dynamic Programming method (\DP), if possible,
\item the continuous version of SDDP (\SDDPc) presented at~\S\ref{sec:sddpc},
\item the discrete version of SDDP (\SDDPd) presented at~\S\ref{sec:sddpd},
\item the DADP method (\DADP).
\end{enumerate}
All these methods produce Bellman functions (optimization stage
described at~\S\ref{sec:num-hydro-summary}), whose quality is evaluated
by the simulation stage of~\S\ref{sec:num-hydro-summary}.
The obtained results are given in Table~\ref{tab:academic}. The lines
``CPU time'' correspond to the time (in minute) needed to compute
the Bellman functions (optimization stage), whereas the lines ``value''
indicate the cost obtained by Monte Carlo on the initial model
(simulation stage, performed using a 100,000 scenarios sample,
except for the 12-Dams valley for which a smaller sample set was used
for computational time constraint).
The comparisons between the different cost values for the same valley
are thus relevant.
For the two methods \SDDPc\ and \DADP, we also give
the upper bound corresponding to the Bellman value obtained at the end
of the optimization stage.
%\vl{je trouve que ce n'est pas clair ce qui
%constitue une borne sup par heuristic, et une borne inf par th\'{e}orie dans le tableau.}

\begin{table}[ht!]
\begin{center}
{\scriptsize
\begin{tabular}{|l|c|c|c|c|c|}
  \hline
  Valley              & \hesp \textbf{4-Dams} \hesp & \hesp \textbf{6-Dams} \hesp  & \hesp \textbf{8-Dams} \hesp & \hesp \textbf{10-Dams} \hesp & \hesp \textbf{12-Dams} \hesp \\
  \hline
  \hline
  \DP\ CPU time       & \bleue{$\mathit{1600}$'}    & \bleue{$\sim\mathit{10^8}$'} & \bleue{$\sim\infty$}        & \bleue{$\sim\infty$}         & \bleue{$\sim\infty$}         \\
  \hline
  \DP\ value          & $3743$                      & \rouge{N.A.}                 & \rouge{N.A.}                & \rouge{N.A.}                 & \rouge{N.A.}                 \\
  \hline
  \hline
  \SDDPd\ CPU time    & \bleue{$\mathit{2}$'}       & \bleue{$\mathit{320}$'}      & \bleue{$\mathit{2250}$'}    & \bleue{$\mathit{133500}$'}   & \bleue{$\sim\infty$}         \\
  \hline
  \SDDPd\ value       & $3737$                      & $7011$                       & $11750$                     & $16920$                      & \rouge{N.A.}                 \\
%  \hline
%  \SDDPd\ upper bound & $\mathit{3738}$             & $\mathit{7028}$              & $\mathit{11640}$            & $\mathit{16815}$             & \rouge{N.A.}                 \\
  \hline
  \hline
  \SDDPc\ CPU time    & \bleue{$\mathit{6}$'}       & \bleue{$\mathit{10}$'}       & \bleue{$\mathit{13}$'}      & \bleue{$\mathit{50}$'}       & \bleue{$\mathit{97}$'}       \\
  \hline
  \SDDPc\ value       & $3742$                      & $7027$                       & $11830$                     & $17070$                      & $\sim17000$                  \\
  \hline
  \SDDPc\ upper bound & $\mathit{3754}$             & $\mathit{7050}$              & $\mathit{11960}$            & $\mathit{17260}$             & $\mathit{19490}$             \\
  \hline
  \hline
  \DADP\ CPU time     & \bleue{$\mathit{7}$'}       & \bleue{$\mathit{12}$'}       & \bleue{$\mathit{18}$'}      & \bleue{$\mathit{24}$'}       & \bleue{$\mathit{22}$'}       \\
  \hline
  \DADP\ value        & $3667$                      & $6816$                       & $11570$                     & $16760$                      & $\sim17000$                  \\
  \hline
  \DADP\ dual value   & $\mathit{3996}$             & $\mathit{7522}$              & $\mathit{12450}$            & $\mathit{17930}$             & $\mathit{20480}$             \\
  \hline
  \hline
  Gap \DADP/\SDDPc    & \rouge{$\mathbf{-2.0}$\%}   & \rouge{$\mathbf{-3.0}$\%}    & \rouge{$\mathbf{-2.2}$\%}   & \rouge{$\mathbf{-1.8}$\%}    & \rouge{$?$}                  \\
  \hline
\end{tabular}
}
\end{center}
\caption{\label{tab:academic}Results obtained by \DP, \SDDPd, \SDDPc\ and \DADP}
\end{table}

We first note that a direct use of \DP\ is only possible for the 4-Dams valley:
it corresponds to the well-known curse of dimensionality inherent to \DP.
The value given by \DP\ is the true optimal cost value for the 4-Dams valley
and can be used as the reference value.
If we go on to the \SDDPd\ method, we observe that it gives rather good
results. The simulation cost obtained for the 4-Dams valley is very close
to the one obtained by \DP, which grounds the quality of the \SDDPd\ method.
But the method suffers from a curse of dimensionality associated to
the control, in the sense that each optimization problem inside \SDDPd\
has to enumerate all possible values of the discrete control, which becomes
prohibitive for large valleys (number of possible controls for a single
dam \emph{power} number of dams). This explains why the method fails on
the 12-Dams valley. Moreover, the \SDDPd\ method does not provide any
upper bound: the cuts are indeed computed by finite differences, so that
we cannot guarantee that the cuts approximation is a lower approximation
of the Bellman functions.
The \SDDPc\ method, although relying on the integrity constraints
relaxation in the optimization stage (hence a not so tight upper bound),
gives excellent results for the 4-Dams valley: we thus elect \SDDPc\ as
the reference method in order to evaluate the \DADP\ method. Note that
the CPU time remains reasonable, the optimization problems inside
\SDDPc\ corresponding to a continuous linear-quadratic formulation (here
solved using the CPLEX commercial solver).\footnote{Note however
that all the methods we are comparing face the curse of dimensionality
associated to the combinatorics of the control during the simulation stage,
as the controls associated to the whole valley have to be enumerated
at each time~$t$ along each scenario. This is the reason why the value
obtained for the 12-Dams valley have been computed using $1000$ scenarios
($100,000$ for the others valleys) and hence not so accurate.}

We now turn to the \DADP\ method. We first notice that the upper bound
given by the method is rather bad (as a consequence of solving a problem
with relaxed coupling constraints  in the optimization stage), but the values
obtained in the simulation stage are correct compared to the ones given
by \SDDPc\ (as indicated by the last line of Table~\ref{tab:academic}).
The most noticeable point is that the CPU time needed for the optimization
stage seems to grow more slowly for \DADP\ than for \SDDPc. This aspect
will be highlighted in \S\ref{sec:num-challenge}.

Let us finally materialize more finely the difference in the results between
\SDDPc\ and \DADP. Beyond average values given in Table~\ref{tab:academic}, Figure~\ref{fig:dam4-payoffs} represents the payoff empirical probability
laws (optimal cost over the time horizon), obtained by the simulation stage
using 100,000 scenarios, for both \SDDPc\ and \DADP. We observe that, although
the expectations are fairly close, the shapes of the two distributions differ
significantly.

\begin{figure}[ht!]
\begin{center}
\includegraphics[scale=0.3]{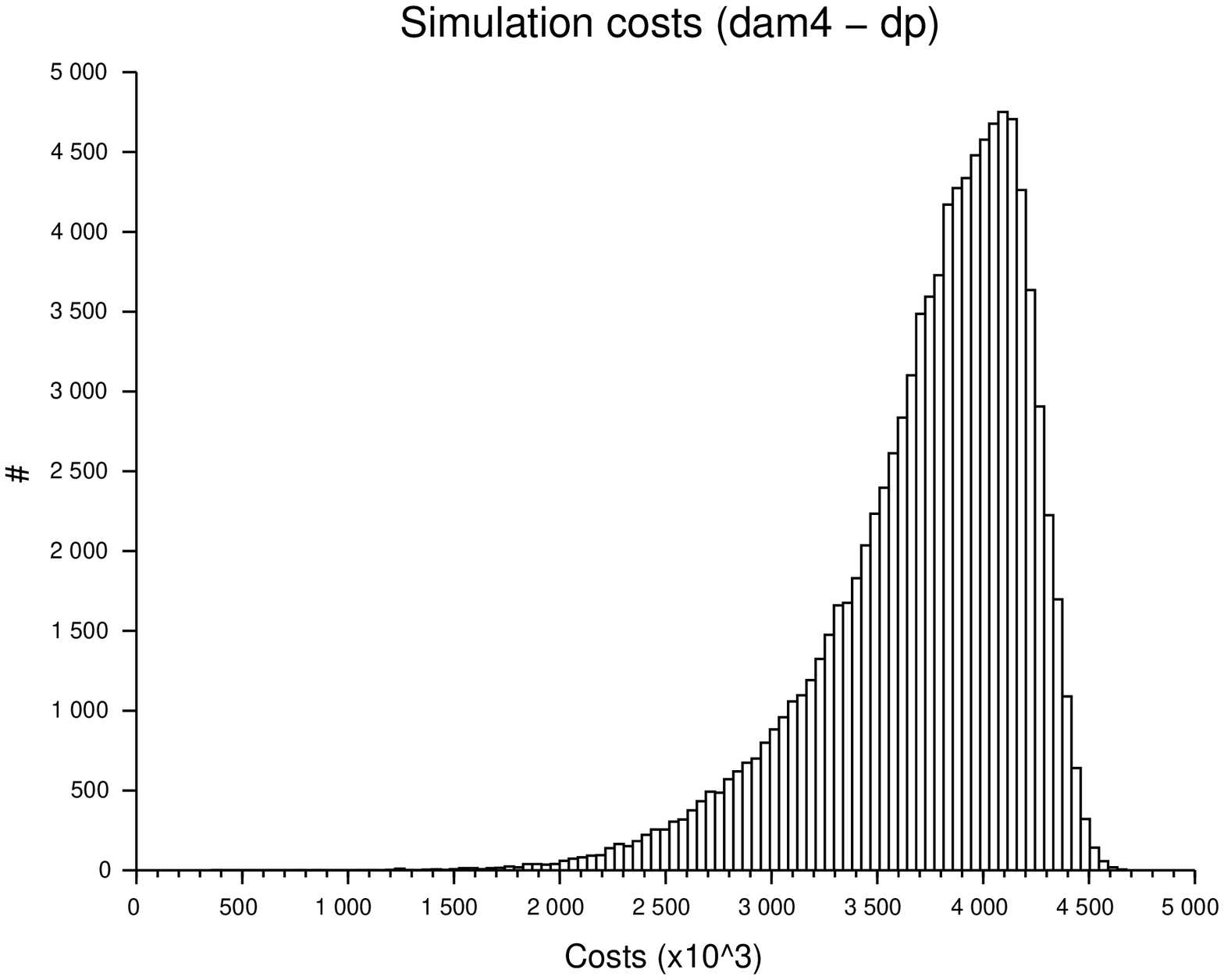}
$\!\!\!\!\!\!\!\!\!\!\!\!\!\!\!\!\!\!\!\!$
\includegraphics[scale=0.3]{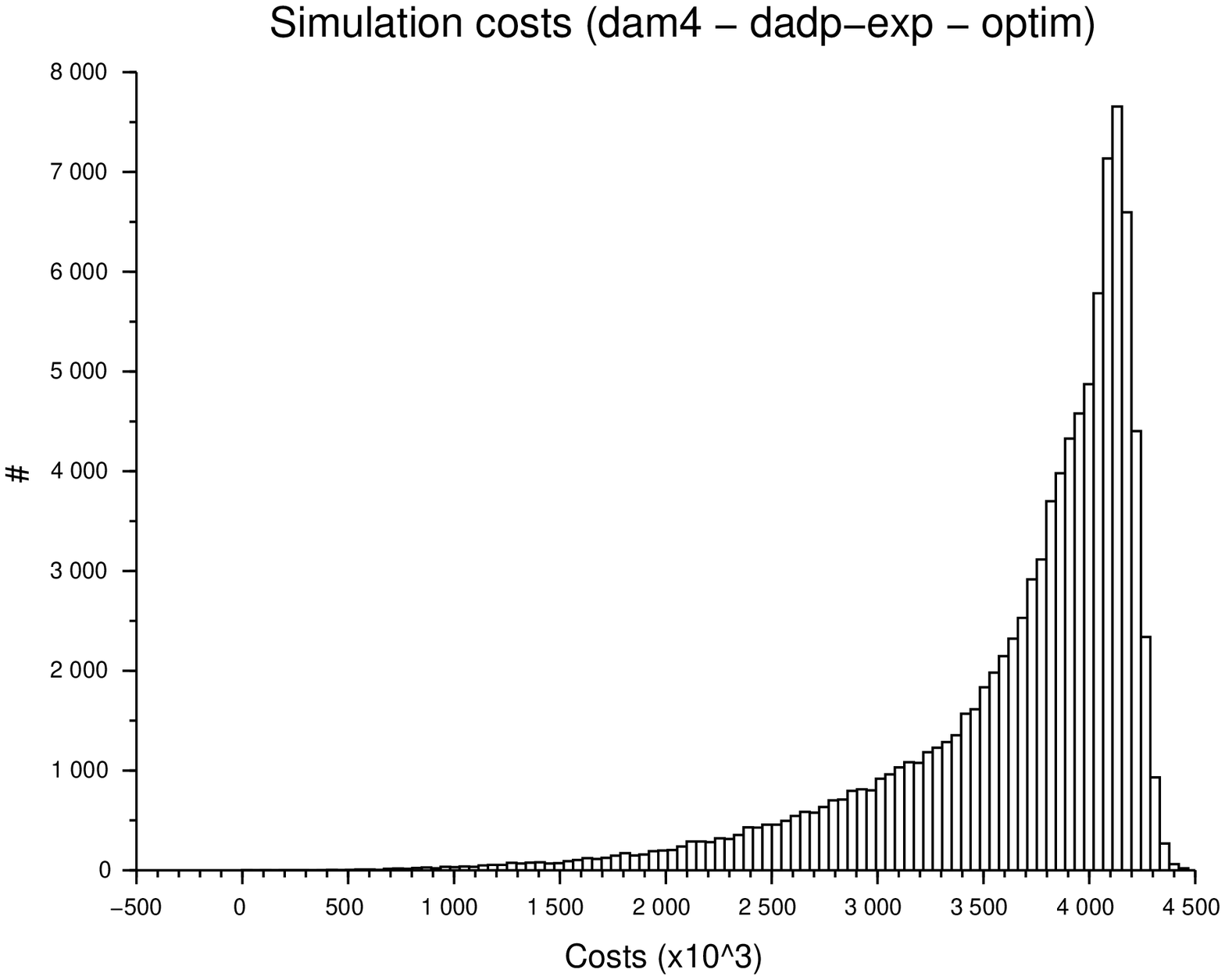}
\end{center}
\caption{4-Dams payoff distributions: \SDDPc\ (left) ~---~ \DADP\ (right)}
\label{fig:dam4-payoffs}
\end{figure}

\subsection{Challenging the curse of dimensionality\label{sec:num-challenge}}

The experiments made in~\S\ref{sec:num-academic} seem to indicate that
\DADP\ is less sensitive to the size of the valley than the \SDDPc\
method. In order to validate this observation, we design a new
collection of academic hydro valleys incorporating
from $14$ up to $30$ dams. It is of course no more possible to perform
the simulation stage for these instances: the combinatorics induced by
the set of possible values of the controls is too large to allow to
simulate the valley behavior along a large set of scenarios. We thus
limit ourselves to the computation of the Bellman functions (optimization
stage). The corresponding results are reported in Table~\ref{tab:curse}.

\begin{table}[ht!]
\begin{center}
{\footnotesize
\begin{tabular}{|l|c|c|c|c|c|}
  \hline
  Valley           & \textbf{14-Dams} & \textbf{18-Dams} & \textbf{20-Dams} & \textbf{25-Dams} & \textbf{30-Dams} \\
  \hline
  \hline
  \SDDPc\ CPU time & $\mathit{210}$'  & $\mathit{585}$'  & $\mathit{970}$'  & $\mathit{1560}$' & $\mathit{2750}$' \\
  \hline
  \DADP\ CPU time  & $\mathit{40}$'   & $\mathit{50}$'   & $\mathit{75}$'   & $\mathit{140}$'  & $\mathit{150}$'  \\
  \hline
\end{tabular}
}
\end{center}
\caption{\label{tab:curse}\SDDPc\ and \DADP\ comparison for large academic valleys}
\end{table}

It appears that the CPU time required for the \DADP\ method grows linearly
with the number of dams, while the growth rate of \SDDPc\ is more or less
exponential. Figure~\ref{fig:cputime} shows how the CPU time varies for
the four methods. As expected, \DP\ is only implementable for small instances,
say up to $5$ dams. \SDDPd\ allows to go a step further, but is limited
by the combinatorial nature of the discrete controls we are considering.
Finally, the limits of \SDDPc\ and \DADP\ have not really be reached,
but \DADP\ displays a linear rate allowing to tackle instances
of even greater size.

\begin{figure}[ht!]
\begin{center}
\includegraphics[scale=0.55]{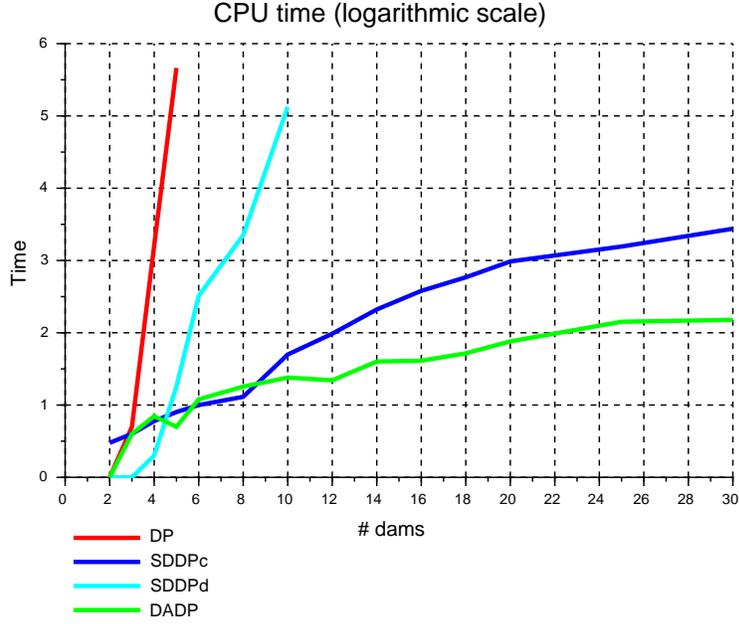}
\end{center}
\caption{CPU time comparison}
\label{fig:cputime}
\end{figure}

\subsection{Results for two realistic valleys\label{sec:num-realistic}}

We finally model two hydro valleys corresponding to existing systems
in France, namely the Vicdessos valley and the Dordogne river
(see Figure~\ref{fig:realistic}).

\begin{figure}[ht!]
\begin{center}
    \hspace{-1.0cm}
  \includegraphics[scale=0.35]{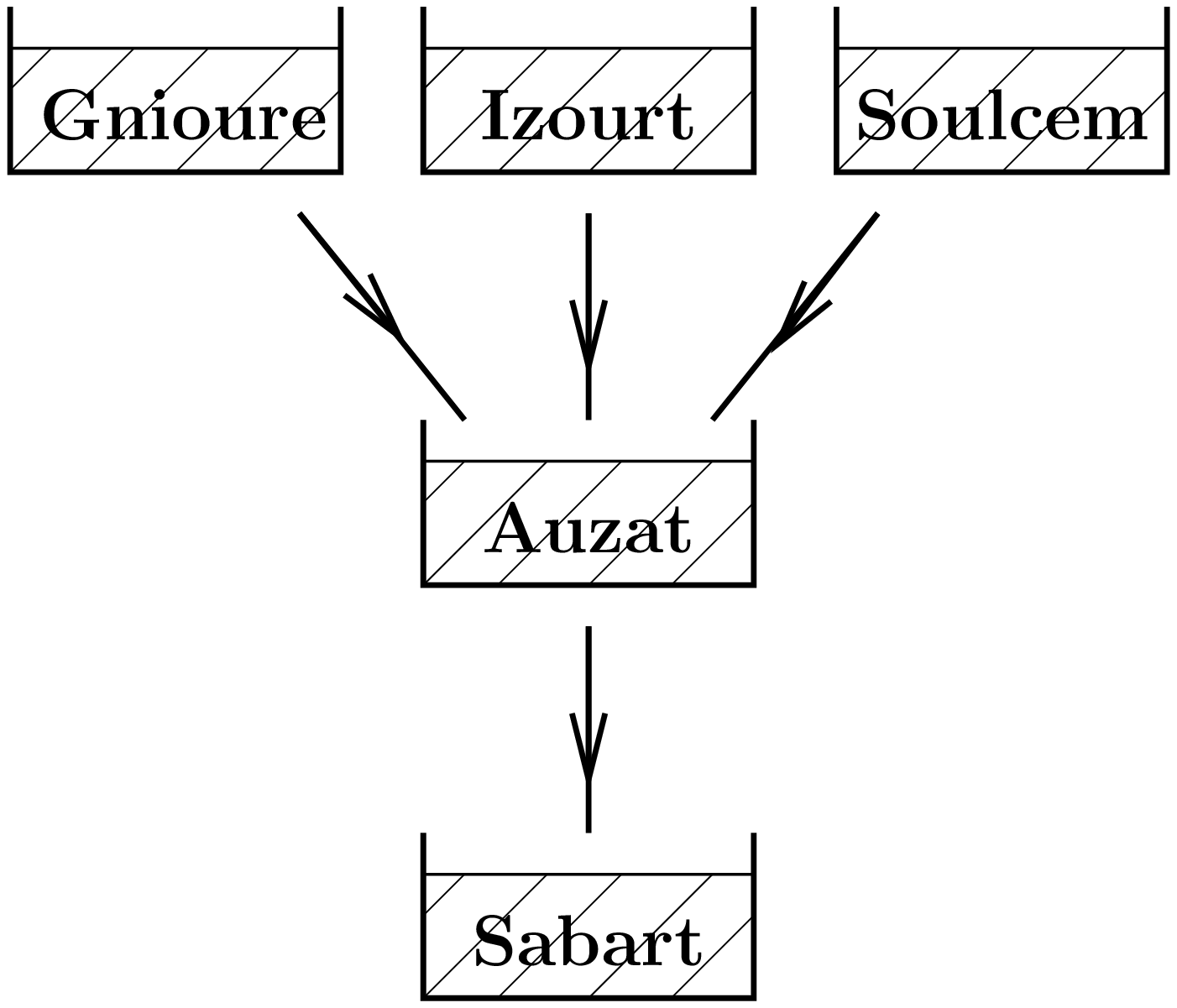} \hspace{2.0cm} $\;$
  \includegraphics[scale=0.35]{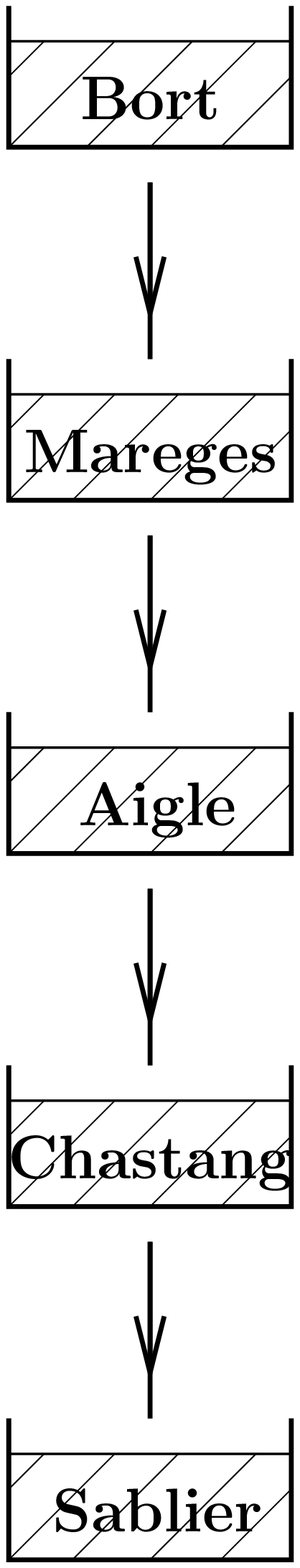}
\end{center}
\vspace{-0.2cm}
\hspace{2.8cm} Vicdessos \hspace{4.0cm} Dordogne
\caption{\label{fig:realistic} Two realistic hydro valleys.}
\end{figure}

The optimization problem is posed again on a time horizon of one year,
with a monthly time step. What mainly differ here from the academic
examples used at \S\ref{sec:num-academic} are the characteristics of
the dams. For example, the Dordogne river valley encompasses large dams
(as ``Bort'' whose capacity is say $400$) and small dams (as ``Mareges''
the capacity of which is equal to $35$, that is, ten times smaller).
This heterogeneity induces numerical difficulties, as the fact to have
at disposal a wide range of possible controls for the small downstream
reservoirs, or the need to use a fine discretization for the state grid
in DP-like methods.
We again assume discrete probability laws with finite support
for the inflows, and we also assume that the available turbine
controls are discrete.

The comparison results of \SDDPc\ and \DADP\ are given in
Table~\ref{tab:realistic}.
\begin{table}[ht!]
\begin{center}
\begin{tabular}{|l|c|c|}
  \hline
  Valley              & \hespa \textbf{Vicdessos} \hespa & \hespa \textbf{Dordogne} \hespa \\
  \hline
  \hline
  \SDDPd\ CPU time    & \bleue{$\mathit{90}$'}           & \bleue{$\mathit{86000}$'}       \\
  \hline
  \SDDPd\ value       & $2234$                           & $21910$                         \\
  \hline
  \SDDPd\ upper bound & $\mathit{2195}$                  & $\mathit{21370}$                \\
  \hline
  \hline
  \SDDPc\ CPU time    & \bleue{$\mathit{9}$'}            & \bleue{$\mathit{17}$'}          \\
  \hline
  \SDDPc\ value       & $2244$                           & $22150$                         \\
  \hline
  \SDDPc\ upper bound & $\mathit{2258}$                  & $\mathit{22310}$                \\
  \hline
  \hline
  \DADP\ CPU time     & \bleue{$\mathit{9}$'}            & \bleue{$\mathit{210}$'}         \\
  \hline
  \DADP\ value        & $2238$                           & $21650$                         \\
  \hline
  \DADP\ dual value   & $\mathit{2286}$                  & $\mathit{22990}$                \\
  \hline
  \hline
  Gap \DADP/\SDDPc    & \rouge{$\mathbf{-0.3}$\%}       & \rouge{$\mathbf{-2.2}$\%}        \\
  \hline
\end{tabular}
\caption{\label{tab:realistic}Results obtained by \SDDPc, \SDDPd\ and \DADP}
\end{center}
\end{table}
As for the academic examples,
\SDDPc\ displays the best results and is therefore used as
the reference. The large number of possible discrete controls
heavily penalizes the \SDDPd\ methods. We also observe that
this large combinatorics is a difficulty even for \DADP,
although the resolution is done dam by dam. Finally,
the gap between \SDDPc\ and \DADP\ remains limited.

%!TEX root = Article_Barrages.tex
%%%%%%%%%%%%%%%%%%%%%%%%%%%%%%%%%%%%%%%%%%%%%%%%%%%%%%%%%%%%%%%%%%%%%%
%                                                                    %
% 5. Conclusion                                                      %
%                                                                    %
%%%%%%%%%%%%%%%%%%%%%%%%%%%%%%%%%%%%%%%%%%%%%%%%%%%%%%%%%%%%%%%%%%%%%%

\section{Conclusion\label{sec:conclusion}}

In this article, we have depicted a method called DADP
which allows to tackle large-scale stochastic optimal control
problems in discrete time, such as the ones found in the field of
energy management. We have mainly presented the practical aspects
of the method, without deepening in the theoretical issues arising
in the foundations of the method.
A lot of numerical experiments have been presented on hydro valley
problems (``chained models''), which complements the ones
already made on unit commitment problems (``flower models'')
\cite{BartyCarpentierGirardeau09,BartyCarpentierCohenGirardeau10}.
The main conclusions are that DADP, on the one hand converges fast,
and on the other hand gives near-optimal results even when using
a ``crude'' relaxation (here a constant information process~$\va{y}$).
More precisely, DADP allows to deal with optimization problems that
are out of the scope of standard Dynamic Programming, and
beats SDDP for very large hydro valleys in terms of CPU time. We
thus hope to be able to implement DADP for very large stochastic
optimal control problems such as the ones encountered in smart
management of urban districts, involving hundreds of houses and
thus hundred of dynamic variables.

The main perspectives that we see beyond this study are to extend it
in two directions. The first direction consists in implementing the DADP
method for general spatial structures (not only ``flower models''
or ``chain model'', but ``smart-grid models'' involving a planar graph).
The second direction is to implement more sophisticated decomposition methods
than price decomposition. On the one hand we want to make use of
decomposition schemes such that resource allocation or interaction
prediction principle \cite{Cohen78}. On the other hand we want to use
augmented Lagrangian based methods such as alternating direction method
of multiplier (ADMM) and proximal decomposition algorithm (PDA) for decomposition
in order to obtain the nice convergence properties of this kind of methods
(see \cite{Lenoir_RAIRO_2017} for a survey).

Finally, let us mention that a theoretical work has begun in order to provide
foundations of the method \cite{TheseLeclere,leclere2013epiconvergence}
it includes conditions for existence of a multiplier in the $L^{1}$ space
when the optimization problem is posed in $L^{\infty}$ and conditions for
convergence of the Uzawa algorithm in $L^{\infty}$. A lot of work remains
to be done on these questions, mainly to relax the continuity assumption in order
to be able to deal with extended functions, and to obtain more general assumptions
ensuring the convergence of Uzawa algorithm.

\paragraph{Acknowledgments}

This research benefited from the support of the FMJH ``Program Gaspard Monge
for optimization and operations research'', and from the support from EDF.

%%%%%%%%%%%%%%%%%%%%%%%%%%%%%%%%%%%%%%%%%%%%%%%%%%%%%%%%%%%%
%                                                          %
% References                                               %
%                                                          %
%%%%%%%%%%%%%%%%%%%%%%%%%%%%%%%%%%%%%%%%%%%%%%%%%%%%%%%%%%%%

%\bibliography{biblio}

\end{document}